\documentclass[12pt]{amsart}
\usepackage[margin=1.2in, letterpaper]{geometry}
\usepackage{mathtools}
\usepackage{amsfonts}
\usepackage{amsthm}
\usepackage{amsmath}
\usepackage{todonotes}
\usepackage{enumerate}
\usepackage{graphicx}
\usepackage{subcaption}
\usepackage[font=small,labelfont=bf]{caption}
\usepackage{url}
\usepackage{hyperref}

\makeatletter
\@namedef{subjclassname@2020}{%
  \textup{2020} Mathematics Subject Classification}
\makeatother

\let\originalleft\left
\let\originalright\right
\renewcommand{\left}{\mathopen{}\mathclose\bgroup\originalleft}
\renewcommand{\right}{\aftergroup\egroup\originalright}
\renewcommand{\Re}{\operatorname{Re}}
\renewcommand{\Im}{\operatorname{Im}}
\newcommand{\Log}{\operatorname{Log}}
\newcommand{\Arg}{\operatorname{Arg}}
\newcommand{\C}{\mathbb{C}}
\newcommand{\R}{\mathbb{R}}

\newcommand{\eps}{\varepsilon}

\AtBeginDocument{%
   \def\MR#1{}
}

\makeatletter
\newtheorem*{rep@theorem}{\rep@title}
\newcommand{\newreptheorem}[2]{%
\newenvironment{rep#1}[1]{%
 \def\rep@title{#2 \ref{##1}}%
 \begin{rep@theorem}}%
 {\end{rep@theorem}}}
\makeatother

\newtheorem{theorem}{Theorem}
\newreptheorem{theorem}{Theorem}
\newtheorem{lemma}{Lemma}
\newreptheorem{lemma}{Lemma}

\begin{document}
\baselineskip=17pt

\title[Newman's conjecture for the extended Selberg class]{A proof of Newman's conjecture for the extended Selberg class}
\author{Alexander Dobner}
\address{Department of Mathematics, UCLA\\
520 Portola Plaza\\
Los Angeles CA 90095}
\email{adobner@math.ucla.edu}
\date{}

\begin{abstract}
Newman's conjecture (proved by Rodgers and Tao in 2018) concerns a certain family of deformations $\{\xi_t(s)\}_{t \in \R}$ of the Riemann xi function for which there exists an associated constant $\Lambda \in \R$ (called the de Bruijn-Newman constant) such that all the zeros of $\xi_t$ lie on the critical line if and only if $t \geq \Lambda$. The Riemann hypothesis is equivalent to the statement that $\Lambda \leq 0$, and Newman's conjecture states that $\Lambda \geq 0$.

In this paper we give a new proof of Newman's conjecture which avoids many of the complications in the proof of Rodgers and Tao. Unlike the previous best methods for bounding $\Lambda$, our approach does not require any information about the zeros of the zeta function, and it can be readily applied to a wide variety of $L$-functions. In particular, we establish that any $L$-function in the extended Selberg class has an associated de Bruijn-Newman constant and that all of these constants are nonnegative. 

Stated in the Riemann xi function case, our argument proceeds by showing that for every $t < 0$ the function $\xi_t$ can be approximated in terms of a Dirichlet series $\zeta_t(s)=\sum_{n=1}^{\infty}\exp(\frac{t}{4} \log^2 n)n^{-s}$ whose zeros then provide infinitely many zeros of $\xi_t$ off the critical line.
\end{abstract}
\maketitle

\section{Introduction}

Let
\[
\xi(s) \coloneqq \frac{1}{2}s(s-1)\pi^{-s/2}\Gamma\left(\frac{s}{2}\right)\zeta(s)
\]
be the Riemann xi function where $\zeta$ denotes the Riemann zeta function. The xi function is entire and satisfies the functional equation $\xi(1-s) = \xi(s)$, from which one may deduce that the zeros of $\xi(s)$ are symmetric about the line $\Re s = \frac{1}{2}$, also known as the critical line. The Riemann hypothesis is equivalent to the statement that all the zeros of the xi function lie on this line.

In order to study the behavior of $\xi$ on the critical line, one may consider the function $z \mapsto \xi\left(\frac{1+iz}{2}\right)$ which has a Fourier representation (see \cite[p.~255]{titchmarsh1986}),
 \begin{equation} \label{eq:fourier}
\xi\left(\frac{1+iz}{2}\right) = \int_{-\infty}^{\infty} \Phi(u) e^{i z u}\, du,
\end{equation}
where
\begin{equation} \label{eq:phisum}
\Phi(u) \coloneqq 4\sum_{n=1}^\infty \left(2\pi^2n^4 e^{9u}-3\pi n^2 e^{5u}\right)e^{-\pi n^2 e^{4u}}
\end{equation}
is a super-exponentially decaying even function. (The evenness of $\Phi$ follows from the fact that it is the Fourier transform of an even, real-valued function.)

In the course of his investigations on the Riemann hypothesis, P\'{o}lya \cite{polya1926} proved the rather remarkable result that if one replaces $\Phi$ in \eqref{eq:fourier} with a certain approximation coming from the first term of \eqref{eq:phisum}, then the resulting ``deformed'' xi function has all of its zeros on the critical line. Moreover, he showed that the average spacing of these zeros is the same as the average spacing of the zeros of the true xi function. This work initiated the study of various other deformations of $\xi$ which come about by modifying $\Phi$ in some way. 

The deformations of interest to us in this paper come from the work of de Bruijn \cite{debruijn1950} and later Newman \cite{newman1976}, who considered the family $\{\xi_t\}_{t\in\R}$ of entire functions defined such that
\[
\xi_t\left(\frac{1+iz}{2}\right) \coloneqq \int_{-\infty}^{\infty} e^{t u^2} \Phi(u) e^{i z u}\, du.
\]
One may verify that each of these functions satisfies the functional equation $\xi_t(1-s)=\xi_t(s)$ as a consequence of the fact that $\Phi$ is even. Furthermore, a result of P\'{o}lya \cite{polya1927} implies that if $\xi_{t_0}$ has all of its zeros on the critical line for some $t_0\in\R$, then the same is true of $\xi_{t}$ for all $t > t_0$. De Bruijn proved that for all $t \geq 1/2$ the zeros of $\xi_t$ all lie on the critical line, and Newman subsequently proved that there is some $t < 1/2$ for which this is not the case. Newman concluded that there exists a real number $\Lambda \leq 1/2$ (now called the \textit{de Bruijn-Newman constant}) with the defining property that $\xi_t$ has all of its zeros on the critical line if and only if $t \geq \Lambda$.

Because $\xi = \xi_0$, it is clear that the Riemann hypothesis is equivalent to the statement that $\Lambda \leq 0$. Newman made the complementary conjecture that $\Lambda \geq 0$, and in doing so he pointed out that this is a quantitative form of the assertion that ``the Riemann hypothesis, if true, is only barely so.'' A progression of lower bounds on $\Lambda$ (see \cite{csordas1988nv}, \cite{teriele1991}, \cite{csordas1991rv}, \cite{norfolk1992rv}, \cite{csordas1993osv}, \cite{csordas1994sv}, \cite{odlyzko2000}, \cite{saouter2011gd}) culminated in a proof of Newman's conjecture by Rodgers and Tao \cite{rodgers2020t} in 2018. Their proof is inspired by previous work of Csordas, Smith, and Varga \cite{csordas1994sv} who noted that zeros of $\xi_t$ exhibit a repulsion effect as $t$ varies. Rodgers and Tao were able to use this repulsion to show via a rather involved argument that if $\Lambda < 0$, then as $t \to 0^-$ the zeros of $\xi_t$ spread out in such a way as to contradict known results about the gaps between zeta zeros. 

In this paper we will give a simpler proof of Newman's conjecture which does not rely on any information about the zeros of the zeta function. In fact, because we use such limited information about $\zeta$, we can state and prove our results in a generalized $L$-function setting. We will set up our generalization in Section~2 by giving the definition of a large class of Dirichlet series $\mathcal{S}^\sharp$ (known as the \textit{extended Selberg class} in the literature) for which it is reasonable to define a de Bruijn-Newman constant. This class contains the Riemann zeta function as well as all Dirichlet $L$-functions associated to primitive characters. The main restriction on membership in $\mathcal{S}^\sharp$ is that every Dirichlet series $F\in \mathcal{S}^\sharp$ must have a corresponding ``completed'' version $\xi^F$ satisfying a certain functional equation. Given such an $F$, we will show that one can always define a set of deformations $\{\xi^F_t\}_{t \in\R}$ analogous to the $\xi_t$ functions we defined above, and we prove the following theorem:

\begin{theorem} \label{setupthm}
For every $F \in \mathcal{S}^\sharp$, there is a real number $\Lambda_F$ such that all the zeros of $\xi^F_t$ lie on the critical line if and only if $t \geq \Lambda_F$. 
\end{theorem}

Previous authors have defined generalized de Bruijn-Newman constants but only for certain restricted classes of $L$-functions (see \cite{stopple2014} for the case of quadratic Dirichlet $L$-functions and \cite[Section 2.4]{andrade2014cm} for the case of certain automorphic $L$-functions). Our definition of $\Lambda_F$ for $F \in \mathcal{S}^\sharp$ subsumes these definitions, and we are also able to prove the analogue of Newman's conjecture in this general case,

\begin{theorem} \label{newmansconj}
$\Lambda_F \geq 0$ for every $F \in\mathcal{S}^\sharp$.
\end{theorem}

We now give the idea behind our proof of Theorem \ref{newmansconj} (restricting to the zeta function case for ease of exposition). The main tool will be an approximation for $\xi_t$ that we establish for every $t<0$ and which we use to locate zeros of $\xi_t$ off the critical line. This approximation is of the form
\begin{equation} \label{eq:approxheuristic}
\xi_t(J_t(s)) \approx \text{[gamma-like factor]} \cdot \zeta_t(s)
\end{equation}
where 
\[
\zeta_t(s) \coloneqq \sum_{n=1}^\infty \exp\left(\frac{t}{4}\log^2 n\right)n^{-s}
\]
is an everywhere absolutely convergent Dirichlet series, and
\[
J_t(s)\coloneqq s+\frac{|t|}{4}\Log\left(\frac{s}{2\pi}\right)
\]
is a nonlinear change of coordinates. If $s$ is restricted to lie in a vertical strip, a crucial feature of our approximation is that it becomes increasingly accurate as the imaginary part of $s$ increases.

Such an approximation hints that for any $t<0$ and any vertical strip $V$, there should be a close correspondence between the zeros of $\zeta_t$ in $V$ and the zeros of $\xi_t$ in the curved region $J_t(V)$. We will prove that this is the case and also that there is some vertical strip where $\zeta_t$ has zeros at arbitrarily high heights. Since the region $J_t(V)$ eventually lies to the right of the the critical line at high enough heights, this argument produces zeros of $\xi_t$ which are off the critical line, thus implying Newman's conjecture.

An example of the correspondence between the zeros of $\zeta_t$ and $\xi_t$ in the case $t=-1$ and $V=\{-.3 \leq \Re s \leq -.2\}$ is depicted in Figure \ref{fig:t1}. The figure also indicates that this correspondence seems to extend quite a bit beyond the range we have just described.  Indeed, one can clearly see that the zeros of $\xi_t$ on the critical line show up as zeros of $\zeta_t$ near the inverse image of the critical line under the $J_t$ map. We note that our results in this paper are insufficient to explain this broader correspondence because the rigorous form of \eqref{eq:approxheuristic} that we prove does not remain accurate if one chooses $J_t(s)$ to lie in a vertical strip rather than $s$.

For large negative $t$ values, a different phenomenon occurs that is readily apparent in Figure \ref{fig:t30} (which depicts zeros for $t=-30$). In this case one sees that the zeros of $\xi_t$ begin to congregate near deterministic curves, and as $t$ becomes increasingly negative more of these curves appear. This phenomenon was originally discovered by Rudolph Dwars (see the comments at \url{terrytao.wordpress.com/2018/12/28}) while doing extensive numerical work on the zeros of $\xi_t$. Our results can be used to give a rigorous explanation for these curves, but since the case of large negative $t$ values is not relevant to Newman's conjecture we do not pursue this direction here.

\begin{figure}
	\centering
	\begin{subfigure}[b]{.46\textwidth}
		\includegraphics[width=\textwidth]{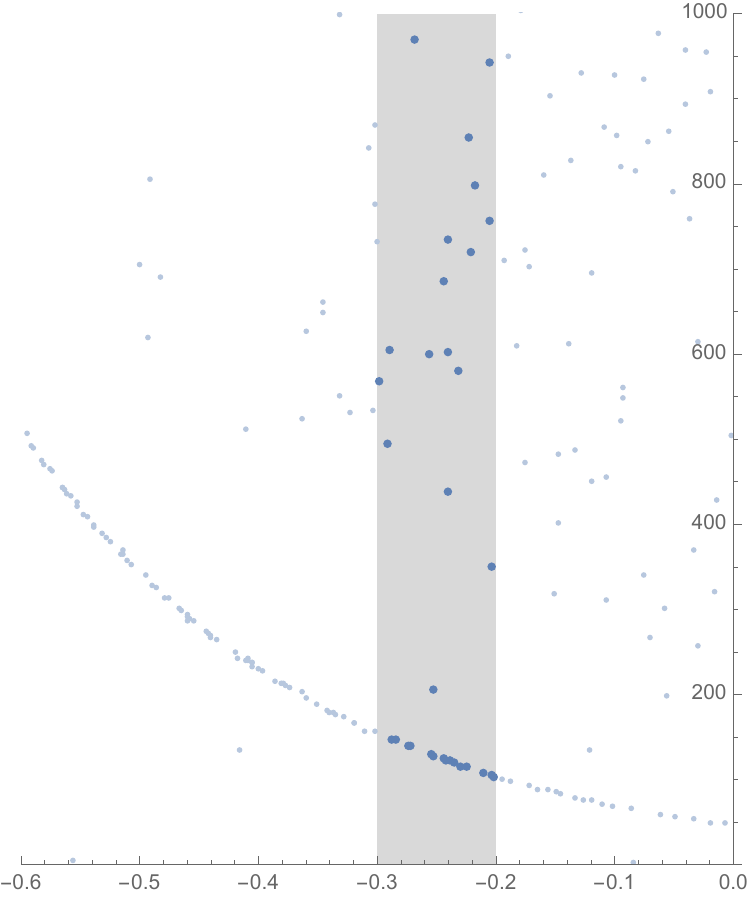}
		\caption{Zeros of $\zeta_{-1}$. }
		\end{subfigure}
	\hfill
	\begin{subfigure}[b]{.46\textwidth}
		\includegraphics[width=\textwidth]{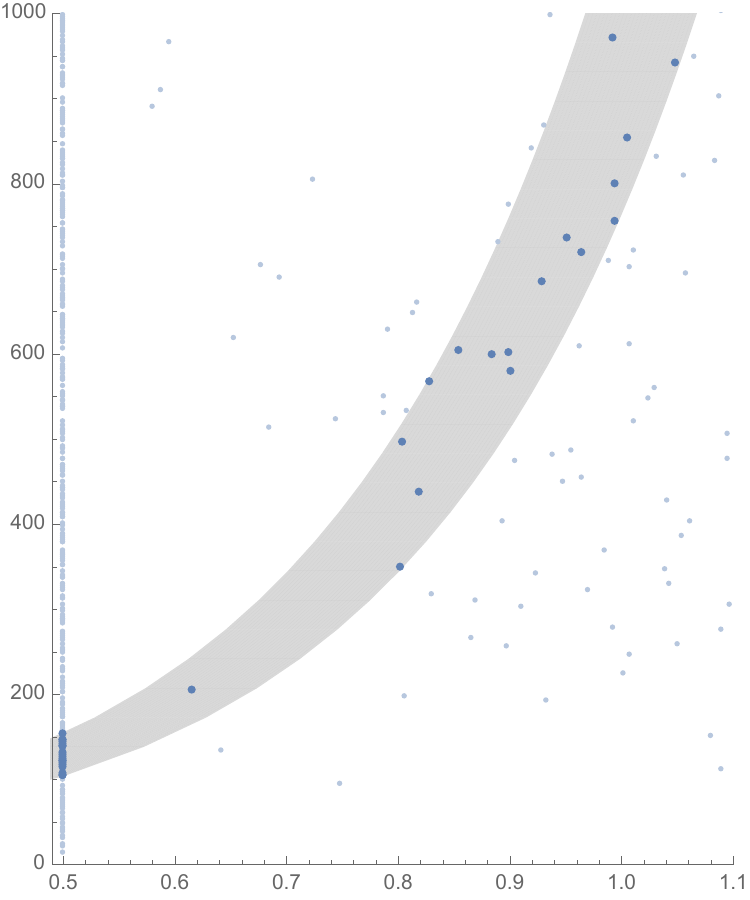}
		\caption{Zeros of $\xi_{-1}$}
	\end{subfigure}
	\caption{The zeros of $\zeta_t$ in any vertical strip are mapped quite precisely under $J_t$ to zeros of $\xi_t$ in the corresponding curved region. This correspondence provides zeros of $\xi_t$ off the critical line. For $\xi_t$ the zeros are symmetric about the critical line, but in (b) we only depict the zeros on and to the right of the line. }
\label{fig:t1}
\end{figure}

\begin{figure}
	\vspace{2em}
	\centering
	\begin{subfigure}[b]{.46\textwidth}
		\includegraphics[width=\textwidth]{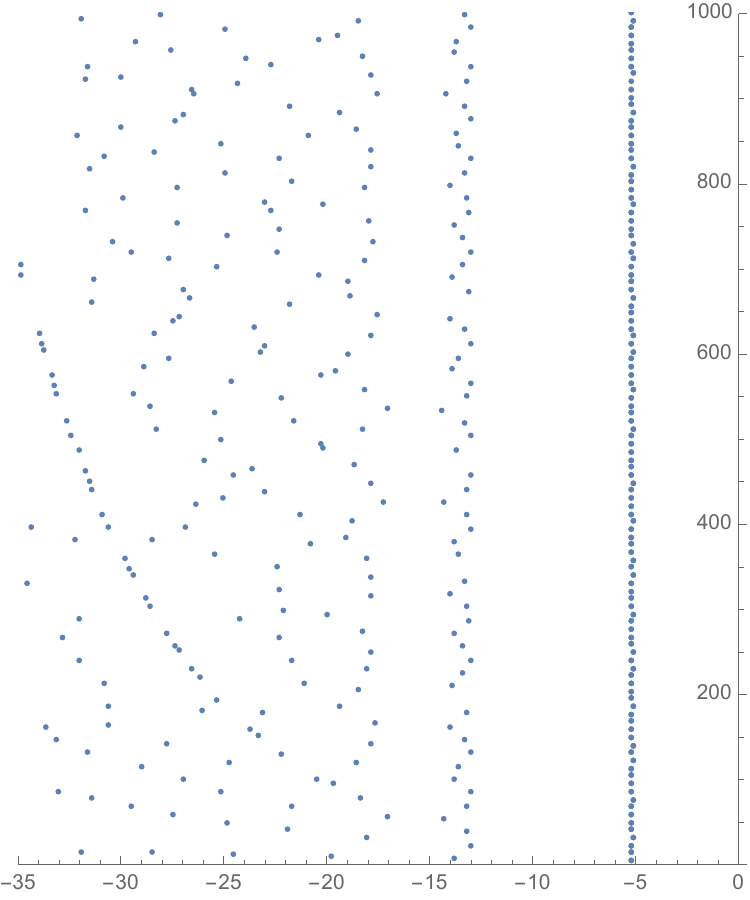}
		\caption{Zeros of $\zeta_{-30}$. }
		\label{fig:ft1}
	\end{subfigure}
	\hfill
	\begin{subfigure}[b]{.46\textwidth}
		\includegraphics[width=\textwidth]{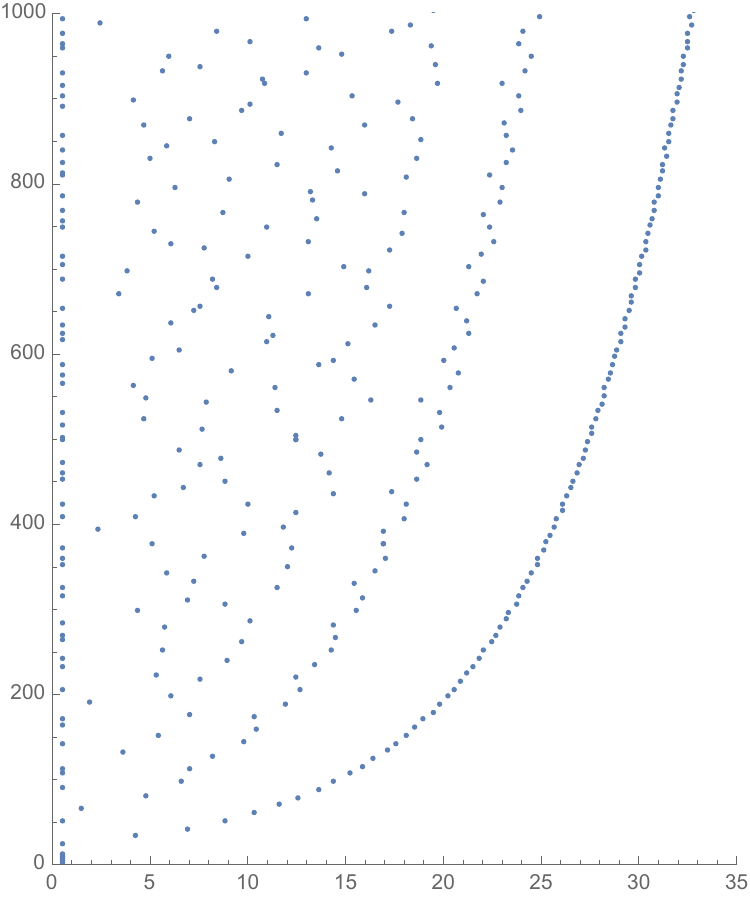}
		\caption{Zeros of $\xi_{-30}$}
	\end{subfigure}
	\caption{For large negative $t$ values, the zeros of $\xi_t$ congregate near curves which are the images of certain vertical lines under the $J_t$ map. On these lines the Dirichlet series $\zeta_t$ is dominated by two consecutive terms of equal magnitude, which leads to the regular pattern of zeros. }
\label{fig:t30}
\end{figure}

To prove the approximation \eqref{eq:approxheuristic} the starting point will be to rewrite $\xi_t$, for $t < 0$, as a certain contour integral of the xi function times a complex Gaussian. This integral is taken over a vertical line in the complex plane, and we are able to estimate it by shifting the contour to the right so that $\zeta(s)$ can be written in terms of its absolutely convergent Dirichlet series. Interestingly, in the $t > 0$ case there is an analogous contour integral representation of $\xi_t$, but the integral is taken over a horizontal line instead, so it is not possible to proceed in the same manner. It turns out that is possible to derive an estimate for $\xi_t$ in terms of partial sums of $\zeta_t$ in this case. This is done in \cite[Thm. 1.3]{polymath2019}, where the estimate is is used to prove an upper bound on the de Bruijn-Newman constant. 
    
One final note we make about our proof is that it reveals that Newman's conjecture holds for completely analytic rather than arithmetic reasons. To highlight this, we mention that our methods can be used to show that if one selects essentially \textit{any} Dirichlet series $F$ and multiplies $F$ by arbitrary $\Gamma$ factors to produce a ``mock'' $\xi^F$ function, then there are corresponding $\xi^F_t$ functions for all $t<0$ and these functions always have zeros off the critical line (though in this completely general case there is no associated de Bruijn-Newman constant because of the lack of a functional equation). Thus, our proof of Newman's conjecture is quite different from the proof of Rodgers and Tao which depends fundamentally on knowledge about the gaps between zeta zeros and hence on the arithmetic structure of the zeta function.

\subsection{Notation}
We will use $\Arg z$ to denote the argument of a complex number, where this value is chosen to lie in the interval $(-\pi,\pi]$. Correspondingly, we use $\Log z$ to denote the standard branch of the complex logarithm (i.e. $\Log z = \log |z| + i \Arg z$), and we define $\sqrt{z} \coloneqq \exp(\frac{1}{2}\Log z)$.

We will use the usual asymptotic notation $X \ll Y$, $Y \gg X$ and $X = O(Y)$ to indicate that there exists a positive constant $C$ such that $|X| \leq CY$. If the constant $C$ depends on other parameters then this will be indicated in the text or by a subscript in the notation (i.e. if $C$ depends on $Z$ then we write $X \ll_Z Y$ or $X=O_Z(Y)$). The notation $X \asymp Y$ indicates $X \ll Y \ll X$, and the notation $X = o_{T\to\infty}(Y)$ indicates that there exists a bound $|X| \leq c(T) Y$ where $c(T) \to 0$ as $T \to \infty$.

In Section 2, we will give the definition of a certain family $\mathcal{S}^\sharp$ of Dirichlet series, and we will subsequently choose an arbitrary $F\in\mathcal{S}^\sharp$ which will remain fixed for the rest of the paper. We will also choose a certain meromorphic function $\gamma$ (depending on $F$) which will be fixed for the rest of the paper. \textit{All} implicit constants in the estimates that we prove will then be allowed to depend on $F$ and $\gamma$. Similarly, if a statement is said to hold for ``sufficiently large $y$'' then the meaning of sufficiently large may depend on $F$ and $\gamma$. Readers who are primarily interested in the Riemann zeta function case of Theorem~\ref{newmansconj} (i.e. Newman's conjecture) may want to skim the definition of $\mathcal{S}^\sharp$ in Section 2 in order to understand the notation and then skip to Section 3 and assume that $F \coloneqq \zeta$.

\subsection{Acknowledgments}
The author thanks Terence Tao for his guidance and for many helpful discussions. The author also thanks the anonymous referee for many comments and suggestions.

\section{The generalized de Bruijn-Newman Constant $\Lambda_F$}
In 1989, Selberg \cite{selberg1989} introduced a definition for a class of Dirichlet series $\mathcal{S}$ now known as the Selberg class. Among other things, Selberg conjectured that all the functions in $\mathcal{S}$ satisfy a corresponding Riemann hypothesis, and the conditions he imposed on $\mathcal{S}$ were chosen with this in mind. 

For the purposes of defining generalized de Bruijn-Newman constants, one of Selberg's conditions (the existence of a functional equation) is completely essential, whereas some of the other conditions (e.g. the existence of an Euler product) are not necessary. Because of this, the Dirichlet series that we consider will be members of a class known as the extended Selberg class $\mathcal{S}^\sharp$ which was defined by Kaczorowski and Perelli \cite{kaczorowski1999p}. This class may be thought of as capturing only the analytic rather than arithmetic properties among those given by Selberg.

A function $F(s)$ is said to be a member of $\mathcal{S}^\sharp$ if it is not identically zero and it satisfies the following conditions:
\begin{enumerate}[(i)]
\item (Dirichlet series) $F$ has a Dirichlet series representation
\[
F(s) = \sum_{n=1}^\infty \frac{a_n}{n^s}
\]
which converges absolutely for all $s$ with $\Re s > 1$. 

\item (Meromorphic continuation) $(s-1)^m F(s)$ is an entire function of finite order for some nonnegative integer $m$.
\item (Functional Equation) Let $m$ be the order of the pole of $F$ at $s=1$ (or let $m=0$ in the case where $F$ has no pole). There exists a function $\gamma(s)$ of the form
\[
\gamma(s) \coloneqq \alpha s^m (s-1)^m Q^s \prod_{i=1}^k \Gamma(\omega_i s + \mu_i)
\]
where $\alpha \in \C \setminus \{0\}$, $Q > 0$, $\omega_i > 0$, and $\mu_i \in \C$ with $\Re \mu_i \geq 0$ such that
\[
\xi^F(s) \coloneqq \gamma(s) F(s)
\]
satisfies
\begin{equation} \label{eq:fneq}
\xi^F(s)=\overline{\xi^F(1-\overline{s})}.
\end{equation}
 \end{enumerate}

\bigskip
\bigskip
Our notation in (iii) differs slightly from the usual notation given in definitions for $\mathcal{S}$ and $\mathcal{S}^\sharp$ (e.g. \cite[p. 160]{kaczorowski2006}) because we include the polynomial factor $s^m (s-1)^m$ in the functional equation. This notation will be convenient for us because we require that the ``completed'' $L$-function $\xi^F$ be an entire function. In our notation it is clear that if $F$ is the Riemann zeta function, then one may choose $\gamma(s) = \frac{1}{2}s(s-1) \pi^{-s/2}\Gamma\left(\frac{s}{2}\right)$ to make $\xi^F$ the usual Riemann xi function.

There are a few straightforward consequences of conditions (i),(ii), and (iii) which we will need later on, so we state these now. Firstly, from (i) it is clear that $|a_n|$ cannot grow too fast. For instance, note that (i) implies, say, $a_n=O(n^2)$. Secondly, since the $\Gamma$ function has no poles in the right half-plane, (ii) and (iii) imply that $\xi^F$ is an entire function. Thirdly, it must be the case that $a_n$ is nonzero for more than one $n$ (i.e. the Dirichlet series representation of $F$ has more than one term). To see this, note that if the Dirichlet series had only a single term then $F$ would have no zeros, but $F$ must have many ``trivial zeros'' coming from the poles of the gamma factors in the functional equation. Lastly, for any vertical strip $V$ in the complex plane, if $s \in V$ and $s$ is bounded away from $1$, then a bound $F(s) = O\left((1+|\Im s|)^A\right)$ holds for some $A\geq 0$ depending on $V$. This follows by applying the Phragm\'{e}n-Lindel\"{o}f principle in the same way one does to get convexity bounds for $\zeta(s)$ in vertical strips. 

We will now state a lemma about the magnitude of $\gamma(s)$ in certain regions, which we will need later. The proof, which is an application of Stirling's formula, will be given in Section \ref{lemmasection}.

\begin{lemma}\label{gammasizelemma}
Let $\gamma(s)$ be one of the functions described in condition (iii) of the extended Selberg class definition. Let $D>0$ and $0 \leq \theta < 1$, and let $s$ be a complex number which is at least unit distance away from the poles and zeros of $\gamma$ and which satisfies $|\Re s| \leq D |\Im s|^\theta$. 
There exist $K,K' >0$ (depending on $\gamma, D,$ and $\theta$) such that
\[
\exp(-K' |\Im s|) \ll |\gamma(s)| \ll \exp(-K |\Im s|),
\]
where the implicit constants may depend on $\gamma, D,$ and $\theta$.
\end{lemma}

One immediate consequence of this lemma (and the Phragm\'{e}n-Lindel\"{o}f bound on $F$) is that $\xi^F(s)$ decays exponentially as $\Im(s)\to\pm\infty$ in any fixed vertical strip. This means that for any $F \in \mathcal{S}^\sharp$, we can perform the same Fourier analytic setup that we did for the Riemann xi function in the introduction. We define
\begin{equation} \label{eq:PhiFdef}
\Phi_F(u) \coloneqq \frac{1}{2\pi}\int_{-\infty}^{\infty} \xi^F\left(\frac{1+ix}{2}\right) e^{-ixu}\, dx
\end{equation}
to be the Fourier transform of $\xi^F$ on the critical line, and likewise we define the family of entire functions $\left\{\xi^F_t\right\}_{t\in\R}$ so that
\begin{equation} \label{eq:xiFtdef}
\xi^F_t\left(\frac{1+iz}{2}\right) \coloneqq \int_{-\infty}^{\infty} e^{t u^2} \Phi_F(u) e^{izu}\, du.
\end{equation}
In the proof of Theorem \ref{setupthm}, we will show that $\Phi_F$ has rapid enough decay for this definition to make sense for all $t\in\R$, and that there exists a real number $\Lambda_F$ for which $\xi^F_t$ has all of its zeros on the critical line if and only if $t \geq \Lambda_F$.

It is important to note that, strictly speaking, for a given Dirichlet series $F \in \mathcal{S}^\sharp$ the choice of $\gamma$ is not unique. For example, $\alpha$ can be scaled by any nonzero real number and the functional equation \eqref{eq:fneq} will still hold. However, it turns out that other than the choice of scaling, $\gamma$ is unique (see \cite[Thm. 2.1]{conrey1993g} which is stated for $\mathcal{S}$ but whose proof only requires the properties we gave for $\mathcal{S}^\sharp$). Since scaling does not affect the locations of the zeros of $\xi^F_t$, there will be no ambiguity in our definition of $\Lambda_F$.

From now on we shall assume that $F$ has been fixed, and that the parameters $\alpha, Q,\omega_i$, and $\mu_i$ have all been chosen as well. All implied constants in error terms will be allowed to depend on these parameters.

In order to prove Theorem \ref{setupthm}, we will rely on the following theorem of de Bruijn extending a previous theorem of P\'{o}lya:
\begin{theorem}[De Bruijn {\cite[Thm. 13]{debruijn1950}}, cf. P\'{o}lya {\cite[Thm. 1]{polya1927}}] \label{debruijnthm}
Let $\phi \colon \R \to \C$ be an integrable function satisfying $\phi(u)=\overline{\phi(-u)}$ and $\phi(u)=O\left(e^{-|u|^b}\right)$ for some $b>2$, and let $G(z) \coloneqq \int_{-\infty}^{\infty} \phi(u) e^{izu}\, du$. For any $t>0$, let $G_t(z)\coloneqq \int_{-\infty}^{\infty} e^{t u^2} \phi(u) e^{izu}\, du$. If all the roots of $G$ lie in the strip $|\Im z| \leq \Delta$, then all the roots of $G_t$ lie in the strip
\[
|\Im z| \leq \left(\max(\Delta^2-2t,0)\right)^{\frac{1}{2}}.
\]
\end{theorem}

\begin{proof}[Proof of Theorem 1]
In order to apply Theorem \ref{debruijnthm}, it is convenient to define a rotated version of $\xi^F$. Let
\[
H(z) \coloneqq \xi^F\left(\frac{1+iz}{2}\right),
\]
so that values of $\xi^F$ on the critical line correspond to values of $H$ on the real line. In this rotated frame of reference, we would like to apply Theorem \ref{debruijnthm} with $H$ corresponding to the function $G$ in the statement of the theorem.

The functional equation \eqref{eq:fneq} implies that $H(x) =\overline{H(x)}$ for any real number $x$, so $H$ is real valued as a function on the real line. Also, $H(x)$ decays exponentially as $x \to\pm\infty$ because of the presence of the $\Gamma$ factors in $\gamma$, so $H$ is a Schwartz function and its Fourier transform is $\Phi_F$ by \eqref{eq:PhiFdef}. Fourier inversion gives,
\[
H(x) = \int_{-\infty}^{\infty} \Phi_F(u) e^{ixu}\, du.
\]

Since $H(x)$ is real valued for all $x\in\R$, its Fourier transform $\Phi_F$ has the conjugate symmetry property $\Phi_F(u)=\overline{\Phi_F(-u)}$. Hence, in order to apply Theorem \ref{debruijnthm} with $\phi \coloneqq \Phi_F$ it will suffice to verify that $\Phi_F$ has the necessary decay as $u \to +\infty$. The decay as $u \to -\infty$ then follows immediately by the symmetry of $\Phi_F$.

To bound $\Phi_F$ we begin by performing the substitution $w \coloneqq \frac{1+ix}{2}$ to rewrite \eqref{eq:PhiFdef} as a complex contour integral,
\[
\Phi_F(u) = \frac{e^u}{\pi i} \int_{\frac{1}{2}-\infty i}^{\frac{1}{2}+\infty i} \xi^F(w) e^{-2u w}\, dw.
\]

After expanding out the definition of $\xi^F$ in the integrand, we wish to interchange the infinite sum coming from the Dirichlet series $F$ with the integral. To justify this, one can first shift the vertical line contour to the right into the half-plane of absolute convergence of $F$ (which can be done because $\xi^F$ decays uniformly exponentially in any vertical strip), and on this new contour the interchange of sum and integral is valid. Hence,
\begin{equation} \label{eq:PhiFsum}
\Phi_F(u) = \frac{\alpha e^u}{\pi i}\sum_{n=1}^{\infty} a_n \int_{2-\infty i}^{2+\infty i} w^m(w-1)^m \prod_{j=1}^{k}\Gamma(\omega_j w+ \mu_j) \left(\frac{n e^{2u}}{Q}\right)^{-w} \, dw.
\end{equation}

It is now helpful to view the above integral as an inverse Mellin transform. Letting $\Psi(w) \coloneqq w^m(w-1)^m \prod_{j=1}^{k}\Gamma(\omega_j w+ \mu_j)$, define $\psi \colon (0,\infty) \to \C$ to be the inverse Mellin transform,
\[
\psi(v) \coloneqq \frac{1}{2\pi i}\int_{2-\infty i}^{2+\infty i} \Psi(w) v^{-w}\, dw.
\]
Then \eqref{eq:PhiFsum} can be rewritten as
\begin{equation} \label{eq:PhiFsum2}
\Phi_F(u) = 2\alpha e^u \sum_{n=1}^{\infty} a_n \psi\left( \frac{n e^{2u}}{Q}\right).
\end{equation}

We can now bound $\psi$ using the following lemma whose proof can be found in Section~\ref{lemmasection}.
\begin{lemma} \label{Rlemma}
Let
\[
h(x) \coloneqq \frac{1}{2\pi i} \int_{1-\infty i}^{1+\infty i} \prod_{j=1}^{k} \Gamma(a_j w + b_j) x^{-w}\, dw
\]
for some $a_j > 0$ and $b_j\in \C$ with $\Re b_j \geq 0$. Then there exists a $\delta>0$ (depending on the $a_j$ and $b_j$ values) such that $h(x) \ll e^{-x^{\delta}}$ for all $x \geq 1$.
\end{lemma}

The function $\psi$ is not quite of the same form as $h$ in the lemma, but by expanding out the polynomial factors in the definition of $\Psi$ and then repeatedly applying the relation
\[
w \Gamma(a w +b) = \frac{1}{a} \Gamma(a w + b+1)-\frac{b}{a}\Gamma(aw+b).
\]
one can write $\psi$ as a linear combination functions to which the lemma does apply. Hence, we can conclude that there exists a $\delta>0$ be such that $\psi(v) \ll e^{-v^{\delta}}$ for all $v \geq 1$. Applying this bound in \eqref{eq:PhiFsum2} we see that for all sufficiently large $u$,
\begin{align*}
\Phi_F(u) &\ll e^u \sum_{n=1}^{\infty} a_n e^{-\left(n^\delta e^{2 \delta u}/Q^\delta\right)} \\
&\ll e^u \sum_{n=1}^{\infty} e^{-n^\delta e^{2 \delta u}/\left(2 Q^\delta\right)}
\end{align*}
where the second inequality follows from the fact that $a_n =O(n^2)= e^{n^{o(1)}}$.

Note that for all $c \in [0, 1/2]$, say, we have a bound $\sum_{n=1}^{\infty} c^{n^\delta} \ll_\delta c$ which we can apply to the sum above to get
\[
\Phi_F(u) \ll e^u e^{-e^{2\delta u}/(2Q^\delta)} \ll e^{-e^{2\delta u}/(3Q^\delta)}
\]
for all $u$ large enough. Hence, $\Phi_F$ has rapid enough decay that Theorem~\ref{debruijnthm} applies.

We are now in a position to define the de Bruijn-Newman constant $\Lambda_F$ associated to $F$. For any $t \in \R$, let
\[
H_t(z) \coloneqq \int_{-\infty}^{\infty} e^{tu^2} \Phi_F(u) e^{izu}\, du
\]
be the functions which are relevant to the conclusions of Theorem \ref{debruijnthm}, and let
\[
\mathcal{Z} \coloneqq \{t \in \R \colon\text{ all the roots of $H_t$ are real}\}.
\]
The functions $\xi^F_t$ defined in \eqref{eq:xiFtdef} are simply the un-rotated versions of $H_t$, and so $\mathcal{Z}$ could just as well be defined as the set of $t$ for which all the zeros of $\xi^F_t$ lie on the critical line.

We now make some observations about the set $\mathcal{Z}$. First we note that $\mathcal{Z}$ is closed. To see this, suppose we have a sequence $\{t_n\} \subset \mathcal{Z}$ converging to a real number $t'$. One may verify (using the extremely rapid decay of $\Phi_F$) that the functions $H_{t_n}$ are entire, and that they converge uniformly on compact subsets to $H_{t'}$. Since the functions $H_{t_n}$ have only real roots, it follows from Hurwitz's theorem (as stated in \cite[Chap. 5 Thm. 2]{ahlfors}, for example) that $H_{t'}$ has only real roots or is identically zero. The latter is impossible because $H_{t'}$ is the Fourier transform of a function which is not identically zero. Hence $t' \in \mathcal{Z}$. So we have shown that $\mathcal{Z}$ is closed.

Next, we note that if $t_0 \in \mathcal{Z}$, then by applying Theorem~\ref{debruijnthm} with $\phi(u)\coloneqq e^{t_0 u^2} \Phi_F(u)$, we may conclude $t \in \mathcal{Z}$ for all $t \geq t_0$. Hence, the only possible choices for $\mathcal{Z}$ are the empty set, all of $\R$, or a half line $\{t \geq \Lambda_F\}$ where $\Lambda_F$ is some real number.

To see that $\mathcal{Z}$ is nonempty, Theorem~\ref{debruijnthm} shows (by choosing $\phi \coloneqq \Phi_F$) that it suffices to check that $H$ has all of its zeros lying in some horizontal strip, or equivalently that $\xi^F$ has all of its zeros in some vertical strip. It is a general fact that any convergent Dirichlet series $F$ has a zero-free half plane because if $s$ has large enough real part, then the first term of $F(s)$ will strictly dominate the sum of all the other terms. This implies $\xi^F$ also has a zero-free half plane, and so by the functional equation for $\xi^F$ all the zeros of $\xi^F$ must lie in a vertical strip.

The fact that $\mathcal{Z}\neq \R$ will be a consequence of the generalized Newman's conjecture which we prove in the next section. Once this is proved it follows that there is a generalized de Bruijn-Newman constant $\Lambda_F$ associated to $F$ for which $\mathcal{Z}=\{t \geq \Lambda_F\}$. The constant $\Lambda_F$ is the unique real number for which $\xi^F_t$ has all of its zeros on the critical line if and only if $t \geq \Lambda_F$.
\end{proof}

\section{Proof of the Generalized Newman's Conjecture}
To prove $\Lambda_F \geq 0$, we will take the direct approach of showing that for every $t < 0$ the function $\xi^F_t$ has zeros off the critical line. As a starting point, we first rewrite our defintion of $\xi^F_t$ for all $t <0$ in terms of a new integral which is essentially a convolution of $\xi^F$ with a Gaussian whose variance is proportional to $|t|$. 

Inserting the definition \eqref{eq:PhiFdef} of $\Phi_F$ into the definition \eqref{eq:xiFtdef} of $\xi^F_t$ and then applying Fubini's theorem (which can be justified for any $t < 0$) gives
\begin{align*}
\xi^F_t\left(\frac{1+iz}{2}\right)
&= \frac{1}{2\pi} \int_{-\infty}^{\infty} \xi^F\left(\frac{1+ix}{2}\right) \int_{-\infty}^{\infty} e^{tu^2} e^{i(z-x)u}\, du\, dx \\
&= \frac{1}{2\sqrt{\pi |t|}} \int_{-\infty}^{\infty} \xi^F\left(\frac{1+ix}{2}\right) e^{\frac{1}{4t}(z-x)^2}\, dx.
\end{align*}
Substituting $s\coloneqq\frac{1+iz}{2}$ and $w\coloneqq\frac{1+ix}{2}$ this becomes
\begin{equation*}
\xi^F_t(s) = \frac{1}{i \sqrt{\pi |t|}} \int_{\frac{1}{2}-\infty i}^{\frac{1}{2}+\infty i} \xi^F(w) e^{\frac{1}{|t|}(s-w)^2}\, dw.
\end{equation*}

Because it will be convenient later on, we will shift the contour in the above integral to the vertical line $\Re w = 2$ where the Dirichlet series for $F$ converges absolutely. Since $\xi^F(s)$ has no poles and decays exponentially in any vertical strip, this shift will not affect the value of the integral. Hence,
\begin{equation} \label{eq:newxiFtdef}
\xi^F_t(s) = \frac{1}{i \sqrt{\pi |t|}} \int_{2-\infty i}^{2+\infty i} \xi^F(w) e^{\frac{1}{|t|}(s-w)^2}\, dw.
\end{equation}
We will make use of this new expression for $\xi^F_t$ in the proof of the next theorem, which will be our main tool for proving the generalized Newman's conjecture.

\begin{theorem} \label{mainthm}
Let $s = x+iy$ with $x\in \R$, $y > 0$, and let $t < 0$. Suppose $|t| \leq C$ and $|x| \leq Cy^{1/4}$ for some positive constant $C$. Then for all $y$ sufficiently large (depending on $C$) the following estimate holds,
\begin{equation} \label{eq:xiFtapprox}
\xi^F_t(J_t(s)) =  \gamma_t(s)\left(F_t(s) +O\left(y^{-1/5} \exp\left(\frac{10}{|t|}\min(x,-2)^2\right)\right)\right)\end{equation}
where
\begin{align} \label{eq:Ftdef}
F_t(s) &\coloneqq \sum_{n=1}^\infty \exp\left(-\frac{|t|}{4}\log^2 n\right) \frac{a_n}{n^s}, \\
J_t(s) &\coloneqq s + \frac{|t|}{2}\log Q + \frac{|t|}{2}\sum_{i=1}^k \omega_i \Log(\omega_i s), \label{eq:Jtdef} \\
\gamma_t(s) &\coloneqq \gamma(s) \exp\left(\frac{1}{|t|}(s-J_t(s))^2\right)
\end{align}
and where the implicit constant in the error may depend on $C$. 
\end{theorem}

One should think of Theorem \ref{mainthm} as showing that $\xi^F_t$ is analogous to $\xi^F$ in that it can be expressed (approximately) as the product of a Dirichlet series and some special additional factors. We stress that the error term in \eqref{eq:xiFtapprox} is not optimal, but for our application to Newman's conjecture, we only require the following qualitative version of the theorem: if $t<0$, $V$ is some vertical strip, and $s=x+iy \in V$ with $y$ sufficiently large, then
\begin{equation}
\xi^F_t(J_t(s)) =  \gamma_t(s)\left(F_t(s) + o_{y \to \infty}(1)\right)
\end{equation}
where the decay of the error term is uniform in $x$ (but not necessarily in $t$). 

The appearance of the Dirichlet series $F_t$ in the theorem can be explained heuristically from \eqref{eq:newxiFtdef} as follows. Suppose one could pull the $\gamma$ factor of $\xi^F(w) = \gamma(w)F(w)$ outside of the integral in \eqref{eq:newxiFtdef}, leaving an integral of the form
\[
\frac{1}{i\sqrt{\pi |t|}} \int_{2-\infty i}^{2+\infty i} F(w) e^{\frac{1}{|t|}(s-w)^2}\, dw.
\]
Interchanging the sum and the integral and then computing the result gives,
\begin{align*}
\sum_{n=1}^{\infty} a_n \frac{1}{i\sqrt{\pi |t|}} \int_{2-\infty i}^{2+\infty i} n^{-w} e^{\frac{1}{|t|}(s-w)^2}\, dw
&= \sum_{n=1}^{\infty}a_n \exp\left(-\frac{|t|}{4}\log^2 n\right) n^{-s} \\
&= F_t(s).
\end{align*}

Further explanation of Theorem \ref{mainthm} including an explanation for the presence of the $J_t$ function and a full proof of theorem will be left for Section 4. We will now show how this theorem implies the generalized Newman's conjecture without much additional work. 

\subsection{Deducing Theorem \ref{newmansconj} from Theorem \ref{mainthm}}
Assume that $t < 0$ is fixed. Theorem~\ref{mainthm} suggests that if $F_t$ has a zero at $s_0$, then $\xi^F_t$ should have a zero near $J_t(s_0)$. In order to make this rigorous, we must somehow deal with the fact that there are error terms present in the correspondence between $F_t$ and $\xi^F_t$ given in the theorem. 

A crucial feature of $F_t$ that we will use is that it is everywhere absolutely convergent. To see this, recall that
\[
F_t(s) \coloneqq \sum_{n=1}^\infty \exp\left(-\frac{|t|}{4}\log^2 n\right) \frac{a_n}{n^s},
\]
and the factor $\exp\left(-\frac{|t|}{4}\log^2 n\right)$ decays faster than $O(n^{-d})$ for any $d$ whereas $a_n=O(n^2)$.

Because $F_t$ is everywhere absolutely convergent it is entire, and it has a property known as \textit{almost periodicity} which was introduced and studied extensively by Bohr \cite{bohr1922}. Roughly speaking, almost periodicity for $F_t$ means that for any vertical strip there is an ample supply of shifts $\tau\in\R$ for which $F_t(s)$ and $F(s + i \tau)$ are uniformly close.

A precise version of one of Bohr's results can be stated as follows,
\begin{theorem}[{\cite[Thm. 1]{bohr1922}}] \label{bohrthm}
Let $G(s) = \sum_{n=1}^{\infty} b_n n^{-s}$ be a Dirichlet series which is absolutely convergent for all $\Re s > \sigma$. For any $\eps > 0$ and $\alpha, \beta \in \R$ such that $\sigma < \alpha < \beta$, there exists a sequence of shifts 
\[
0 < \tau_1 < \tau_2 < \cdots
\]
satisfying
\[
\liminf_{m \to \infty} (\tau_{m+1}-\tau_m) > 0 \text{\quad and \quad } \limsup_{m \to \infty} \frac{\tau_m}{m} < \infty
\]
such that
\[
\left| G(s) - G(s+i \tau_m) \right | < \eps
\]
for all $\alpha \leq \Re s \leq \beta$.
\end{theorem}

Because $F_t$ is everywhere absolutely convergent, the above theorem holds for $F_t$ for any choice of $\alpha < \beta$. 

In order to proceed with the proof, we suppose for the moment that one is able to locate a single zero $s_0$ of $F_t$. Let $C$ be a circle centered at $s_0$ of radius $r$ chosen so that $C$ does not pass through any other zeros of $F_t$, and let
\[
\delta \coloneqq \min_{s \in C} |F_t(s)| > 0.
\]

Let $V$ denote the vertical strip $\left\{|\Re s - \Re s_0 | \leq r\right\}$ and recall that the qualitative form of Theorem \ref{mainthm} states
\begin{equation} \label{eq:xiFtapproxV}
\xi^F_t(J_t(s)) = \gamma_t(s) \left( F_t(s)+ o_{y\to\infty}(1) \right)
\end{equation}
for $s=x+iy\in V$, $y$ sufficiently large, where the decay of the error term is uniform in $x$.

Let
\[
h(s) \coloneqq \xi^F_t(J_t(s))/\gamma_t(s),
\]
and note that $h$ is analytic in the upper half plane, and
\begin{equation} \label{eq:happrox}
h(s) = F_t(s)+o_{y\to\infty}(1)
\end{equation}
for $s=x+iy \in V$, $y>0$ by \eqref{eq:xiFtapproxV}.

Applying Theorem \ref{bohrthm} to $F_t$ on the strip $V$ and taking $\eps \coloneqq \delta/3$, let $0 < \tau_1 < \tau_2 < \ldots$ be the resulting sequence of shifts.

By picking $\tau_m$ which is sufficiently large, one can ensure that for any $s$ on the circle $C$ the shift $s+i \tau_m$ will have a large enough imaginary part that the error term in \eqref{eq:happrox} when evaluating $h(s+i \tau_m)$ is uniformly less than $\delta/3$. Consequently, for all $s \in C$ we have
\[
|h(s+i\tau_m) - F_t(s)| \leq |F_t(s+i\tau_m) - F_t(s)|+\frac{\delta}{3} \leq \frac{2\delta}{3} < |F_t(s)|,
\]
so by Rouch\'{e}'s theorem $h$ has a zero inside the shifted circle $C+i \tau_m$.

By taking larger and larger shifts $\tau_m$, the argument above yields an infinite collection of zeros of $h$ which are arbitrarily high up the strip $V$. From the definition of $h$ and the fact that poles of $\gamma_t$ only exist up to some bounded height, we can deduce that $\xi^F_t(J_t(s))$ also has an infinite collection of zeros and these zeros exist arbitrarily high up the strip $V$. Hence, $\xi^F_t$ has infinitely many zeros within the curved region $J_t(V)$ at arbitrarily high heights. Past a certain height $J_t(V)$ lies completely to the right of the critical line so in particular we have shown that $\xi^F_t$ has zeros off the critical line.

This completes the deduction of the generalized Newman's conjecture from Theorem \ref{mainthm} and the assumption that $F_t$ has at least one zero. We will assert this latter fact as a lemma for now, and the proof will be given in Section \ref{lemmasection}.
\begin{lemma} \label{Ftzerolemma}
For any $t < 0$, $F_t$ has a zero.
\end{lemma}

\section{Proof of Theorem \ref{mainthm}}
We are interested in estimating $\xi^F_t(J_t(s))$, so by \eqref{eq:newxiFtdef} we know
\[
\xi^F_t(J_t(s)) = \frac{1}{i \sqrt{\pi |t|}} \int_{2-\infty i}^{2+\infty i} \xi^F(z) e^{\frac{1}{|t|}(J_t(s)-z)^2}\, dz
\]
for all $t<0$. By inserting the definition $\xi^F(z) = \gamma(z) \sum_{n=1}^{\infty} a_n n^{-z}$ and then interchanging the sum and integral (which can be justified since the contour is in the half-plane of absolute convergence of the Dirichlet series $F$) we get
\[
\xi^F_t(J_t(s)) = \sum_{n=1}^{\infty} a_n B_{t,n}(s)
\]
where
\begin{equation} \label{eq:Btndef}
B_{t,n}(s) \coloneqq \frac{1}{i \sqrt{\pi |t|}} \int_{2-\infty i}^{2+\infty i} \gamma(z) e^{\frac{1}{|t|}(J_t(s)-z)^2} n^{-z}\, dz.
\end{equation}

Our goal is to estimate $B_{t,n}(s)$ in the regime of $s$ and $t$ values we are interested in. We will prove the following lemma:

\begin{lemma} \label{mainlemma}
Let $t$, $s=x+iy$, and $C$ satisfy the same restrictions as in Theorem \ref{mainthm}. Then for all $y$ sufficiently large (depending on $C$), the following estimates for $B_{t,n}(s)$ hold for small, medium, and large $n$ respectively (where the implicit constants may depend on $C$):
\begin{enumerate}[(i)]
\item For $1 \leq n \leq \exp(y^{1/3}/|t|)$,
\[
B_{t,n}(s) = \gamma_t(s) \exp\left(-\frac{|t|}{4}\log^2 n\right)n^{-s}(1+O(y^{-1/5})).
\]
\item For $1 \leq n \leq \exp(y^{3/5}/|t|)$,
\[
B_{t,n}(s) =O\left( |\gamma_t(s)| \exp\left(-\frac{|t|}{8}\log^2 n\right)n^{-x}\right).
\]
\item For $n > \exp(y^{3/5}/|t|)$,
\[
B_{t,n}(s) = O\left(\exp\left(-\frac{|t|}{10} \log^2 n\right)\right).
\]
\end{enumerate}
\end{lemma}

Before giving the proof we will summarize our method. Looking at the integral in \eqref{eq:Btndef}, one may note that the integrand exhibits very rapid decay along the vertical line contour. Hence, a reasonable approach to estimate $B_{t,n}(s)$ is to split the contour into a finite segment where the mass of the integrand is concentrated, and a tail whose size can be crudely bounded. It then remains to estimate the integral over the finite segment. Unfortunately, this problem is still nontrivial because the integrand can be quite oscillatory on this segment. One way to proceed is to use the method of stationary phase to estimate the resulting oscillatory integral (e.g. see \cite[Ch. 4]{titchmarsh1986} for useful lemmas in this direction). We will instead use the complex analytic analogue of this method known as the method of steepest descent. The basic principle is to start by shifting the contour so that the mass of the integrand is concentrated somewhere where the integrand is not so oscillatory. Upon doing so, one may then estimate the resulting integral and get good error terms without needing to be too sophisticated about handling cancellation.

We are now in a position to describe the significance of the $J_t(s)$ function. This function has been defined such that when performing the contour shift to estimate the value of $\xi^F_t(J_t(s))$, the region where the integrand oscillates negligibly is near the point $s$ (although the precise shift will depend on the value of $n$ as well). 

We now give a heuristic argument for Lemma \ref{mainlemma} which will also serve as a sketch to be made rigorous. We will skip over any details about contour shifting for now.

The first step is to locally approximate the integrand of $B_{t,n}(s)$ near $s$ because it turns out that this is approximately where the dominant contribution is. We will see in an upcoming lemma that for all $z$ suitably close to $s$ one has
\begin{equation} \label{eq:gheur}
\gamma(z) \approx \gamma(s) \exp\left(\left(\log Q+\sum_{i=1}^k \omega_i \Log(\omega_i s)\right)(z-s)+\frac{\sum_{i=1}^k \omega_i}{2 s}(z-s)^2\right)
\end{equation}
where this is essentially coming from a Taylor expansion of $\Log \gamma(z)$ around $s$.
A similar Taylor expansion of the Gaussian factor in the integrand of $B_{t,n}(s)$ yields
\begin{equation} \label{eq:expid}
e^{\frac{1}{|t|}(J_t(s)-z)^2} = e^{\frac{1}{|t|}(J_t(s)-s)^2} \exp\left(\frac{2}{|t|}(s-J_t(s))(z-s)+\frac{1}{|t|}(z-s)^2\right)
\end{equation}
where in this case the Taylor expansion is exact because the phase is simply a quadratic polynomial.

Taking the product of \eqref{eq:gheur} and \eqref{eq:expid}, the constant term in the resulting Taylor expansion is
\[
\gamma(s) e^{\frac{1}{|t|}(J_t(s)-s)^2} = \gamma_t(s),
\]
and the coefficient of the $(z-s)$ term is
\[
\log Q+\sum_{i=1}^k \omega_i \Log(\omega_i s)+\frac{2}{|t|}(s-J_t(s))=0
\]
where this equality holds because we have chosen the definition of $J_t(s)$ to make it so. 

Putting this information together, we get the heuristic approximation
\begin{align*}
B_{t,n}(s) \approx \gamma_t(s) \frac{1}{i\sqrt{\pi |t|}}\int_{2-\infty i}^{2+\infty i} \exp\left(A(z-s)^2\right)n^{-z}\, dz
\end{align*}
where $A \coloneqq \frac{1}{|t|}+\frac{\sum_{i=1}^k \omega_i}{2s}$.

The remaining integral can now be computed directly. By completing the square,
\begin{multline*}
\exp\left(A(z-s)^2\right)n^{-z} = \exp\left(-\frac{1}{4A}\log^2 n\right)n^{-s} \\ \times \exp\left(A\left(z-\left(s+\frac{1}{2A}\log n\right)\right)^2\right)
\end{multline*}
which means
\begin{multline} \label{eq:Bheur}
B_{t,n}(s) \approx \gamma_t(s) \exp\left(-\frac{1}{4A}\log^2 n\right)n^{-s} \\
\times \frac{1}{i\sqrt{\pi |t|}}\int_{2-\infty i}^{2+\infty i} \exp\left(A(z-(s+\frac{1}{2A}\log n))^2\right)\, dz.
\end{multline}

The value of this last integral is just $i\sqrt{\pi/A}$. Inserting this and then removing all instances of $A$ in the resulting expression by using the approximate equality $A \approx \frac{1}{|t|}$ gives the desired main term in (i). 

Before giving a rigorous version of the heuristic calculations above, we first list several lemmas that we will need. The proofs of these lemmas will be given in Section \ref{lemmasection}. The first is a rigorous version of \eqref{eq:gheur}:
\begin{lemma} \label{gapproxlemma}
For any $\eps > 0$, define the region $S_\eps \coloneqq \left\{|\Arg w| < \pi - \eps, |w|>\eps\right\}$. Let $z,z_0 \in S_\eps$ and $|z-z_0| \leq D |z_0|^{2/3}$ for some $D > 0$. Then for all $z, z_0$ that are at least unit distance away from the poles and zeros of $\gamma$, we have the estimate
\begin{multline*}
\gamma(z) = \gamma(z_0) \exp\left(\left(\log Q+\sum_{j=1}^{k} \omega_j \Log(\omega_j z_0)\right)(z-z_0)+\frac{\sum_{j=1}^k \omega_j}{2z_0} (z-z_0)^2\right) \\
\times \left(1+O\left(\frac{1+|z-z_0|}{|z_0|}+\frac{|z-z_0|^3}{|z_0|^2}\right)\right)
\end{multline*}
where the implicit constant may depend on $\eps$ and $D$.
\end{lemma}

The above lemma can be used to approximate the integrand of $B_{t,n}(s)$ when $z$ is relatively close to $s$. It will also be useful to have an upper bound on the integrand when $z$ is far away from $s$. 
\begin{lemma} \label{integrandboundlemma}
Let $t$, $s=x+iy$, and $C$ satisfy the same hypotheses as in Theorem \ref{mainthm}. Let $n \leq \exp(y^{3/5}/|t|)$ (i.e. the small/medium case of Lemma~\ref{mainlemma}) and let
\[
\mathcal{I}(z) \coloneqq \gamma(z)e^{\frac{1}{|t|}(J_t(s)-z)^2}n^{-z}
\]
be the integrand of \eqref{eq:Btndef}. Then for all $y$ sufficiently large (depending on $C$) the following estimates hold (where the implicit constants may also depend on $C$):
\begin{enumerate}[(i)]
\item If $|x-\Re z| \leq y^{3/5}$ and $\frac{1}{2} y^{2/3} \leq |y-\Im z| \leq 2y^{2/3}$ then 
\[
\mathcal{I}(z) \ll \exp\left(-\frac{y^{4/3}}{10|t|}\right).
\]
\item If $\Re z = 2$ and $|y-\Im z| \geq y^{2/3}$ then
\[
\mathcal{I}(z) \ll \exp\left(-\frac{1}{2|t|}(y-\Im z)^{2}\right).
\]
\end{enumerate}
\end{lemma}

We will also need a weak bound on $\gamma$ at some point, 
\begin{lemma} \label{gboundlemma}
There is some $K>0$ such that $|\gamma(z)| \leq \exp(K (\Re z)^{1.1})$ uniformly for any $z$ with $\Re z \geq 1$.
\end{lemma} 

\begin{proof}[Proof of Lemma \ref{mainlemma}]
Recall that $s=x+iy$, and we assume that $y > 0$,  $|x| \leq Cy^{1/4}$ and $-C < t < 0$. The lemma is only asserted to hold when $y$ is sufficiently large, and throughout the proof we will frequently assume that $y$ is large enough (depending on $C$) to make various statements hold. We also allow all implicit constants to depend on $C$.

We start by proving (i) and (ii) using the heuristic calculations given above as framework. These are the cases of small and medium $n$ where $|t| \log n \leq y^{1/3}$ and $|t| \log n \leq y^{3/5}$ respectively. We will not need to distinguish between these two cases until the very end of the proof, so for now we will just assume that the latter bound holds. 

Recall that in the heuristic calculations, the final integral in \eqref{eq:Bheur} was a complex Gaussian centered at the point $s+\frac{1}{2A} \log n$ where $A = \frac{1}{|t|} + \frac{1}{2s}\sum_{j=1}^{k}\omega_j$. We will shift our contour to pass through this point. The shifted contour we select is the piecewise linear contour passing though $2-\infty i$, $2 + (y-y^{2/3})i$, $s+\frac{1}{2A}\log n - y^{2/3} i$, $s+\frac{1}{2A}\log n + y^{2/3} i$, $2+(y+y^{2/3})i$, and $2+\infty i$. Let $V_1$, $H_1$, $M$, $H_2$, and $V_2$ denote these linear pieces respectively.

We will make frequent use of the estimate 
\begin{equation} \label{eq:Aest}
A = \frac{1}{|t|}\left(1+O\left(\frac{|t|}{y}\right)\right) \asymp \frac{1}{|t|},
\end{equation}
starting with the fact that this implies $\frac{1}{2A}\log n \ll y^{3/5}$. From this and our assumption $x \ll y^{1/4}$, it follows that the contour shift we are performing occurs completely within a $O(y^{2/3})$ radius ball around $s$. One consequence of this is that for all large $y$ the contour shift does not pass over any poles of the integrand (which come from the $\Gamma$ factors) because the imaginary parts of these poles are uniformly bounded above. Hence, by Cauchy's theorem
\[
\int_{2-\infty i}^{2+\infty i} = 
\int_{V_1}+\int_{H_1}+\int_{M}+\int_{H_2}+\int_{V_2}
\]
where the $M$ integral will contribute the main term of our estimate, and the other integrals are negligible in comparison. We will apply Lemma \ref{gapproxlemma} to estimate the $M$ integral and Lemma \ref{integrandboundlemma} to bound the other four.

\subsubsection*{M integral}
As noted above, the segment $M$ lies completely within a circle of radius $O(y^{2/3})$ around $s$, so Lemma \ref{gapproxlemma} gives an approximation for $\gamma(z)$ in terms of $\gamma(s)$ for every $z\in M$. Applying the lemma and then manipulating the integrand in the same way that we did in our heuristic calculations (but now carrying along the error terms coming from Lemma~\ref{gapproxlemma}), the integral over $M$ becomes 
\begin{multline} \label{eq:Mapprox}
 \gamma_t(s) \exp\left(-\frac{1}{4A}\log^2 n\right)n^{-s}
 \frac{1}{i\sqrt{\pi |t|}}\\
 \times \int_M \exp\left(A(z-(s+\frac{1}{2A}\log n))^2\right)\left(1+O\left(\frac{1+|z-s|}{|s|}+\frac{|z-s|^3}{|s|^2}\right)\right)\, dz.
\end{multline}

Parametrizing $M$ as $z(u) = s+ \frac{1}{2A} \log n + ui$ for $u=-y^{2/3}$ to $u=y^{2/3}$, and then using the fact that $|z(u)-s| \leq |u| + \left|\frac{1}{2A} \log n \right| \ll |u| + y^{3/5}$ and $|s| \asymp y$ to simplify the error terms we deduce that \eqref{eq:Mapprox} equals
\begin{equation} \label{eq:Mapprox2}
 \gamma_t(s) \exp\left(-\frac{1}{4A}\log^2 n\right)n^{-s} \frac{1}{\sqrt{\pi |t|}}\int_{-y^{2/3}}^{y^{2/3}} \exp\left(-Au^2\right) \left(1+O\left(\frac{1+|u|^3}{y^{1/5}}\right)\right)\, du.
\end{equation}

The integral of the error term can be bounded via the triangle inequality,
\begin{align*}
\int_{-y^{2/3}}^{y^{2/3}} \exp\left(-Au^2\right) O\left(\frac{1+|u|^3}{y^{1/5}}\right) \, du &\ll y^{-1/5} \int_{-\infty}^{\infty} \exp\left(-\frac{1}{2 |t|}u^2\right) \left(1+|u|^3\right) \, du \\
&\ll \sqrt{|t|} y^{-1/5},
\end{align*}
where we have used the fact that $\Re A \geq \frac{1}{2|t|}$ which comes from \eqref{eq:Aest} and taking $y$ sufficiently large. 

Estimating the main term now we see that
\begin{align*}
\int_{-y^{2/3}}^{y^{2/3}} \exp\left(-Au^2\right) \, du &= 
\int_{-\infty}^{\infty} \exp\left(-Au^2\right) \, du +O\left(\int_{y^{2/3}}^{\infty} \exp\left(-(\Re A)u^2\right) \, du\right) \\
&= \sqrt{\frac{\pi}{A}}+O\left(|t| \exp\left(-\frac{y^{4/3}}{2|t|}\right)\right),
\end{align*}
where the second line follows from the first by using $\Re A \geq \frac{1}{2|t|}$ again.

Inserting these estimates into \eqref{eq:Mapprox2} and using $\frac{1}{\sqrt{A}} = \sqrt{|t|}(1+O(y^{-1}))$, which follows from \eqref{eq:Aest}, we see that the value of the $M$ contour integral is
\[
=\gamma_t(s) \exp\left(-\frac{1}{4A}\log^2 n\right)n^{-s} (1+O(y^{-1/5}))
\]
which we will see later is the main term of our estimate.

\subsubsection*{$H_1$ and $H_2$ integrals}
The $H_1$ integral is
\[
\frac{1}{i \sqrt{\pi |t|}} \int_{H_1} \gamma(z) e^{\frac{1}{|t|}(J_t(s)-z)^2} n^{-z}\, dz,
\]
so applying the $ML$-inequality gives the bound
\begin{equation} \label{eq:MLbound}
\ll \frac{1}{\sqrt{|t|}} |H_1| \max_{z\in H_1} \left|\gamma(z) e^{\frac{1}{|t|}(J_t(s)-z)^2} n^{-z}\right|
\end{equation}
where $|H_1|$ denotes the length of the contour.

Recall that $H_1$ is the line segment from $2+(y-y^{2/3})i$ to $s+\frac{1}{2A} \log n -y^{2/3} i$, and since $\left| \frac{1}{2A}\log n \right| \leq y^{3/5} \leq y^{2/3}/2$ for all sufficiently large $y$, the hypotheses of Lemma~\ref{integrandboundlemma}(i) are satisfied, so
\[
\max_{z\in H_1} \left|\gamma(z) e^{\frac{1}{|t|}(J_t(s)-z)^2} n^{-z}\right| \ll \exp\left(-\frac{y^{4/3}}{10|t|}\right).
\]

Also,
\[
|H_1| = \left|x+\frac{1}{2A} \log n -2 \right| \ll y^{1/4}+y^{3/5}+2 \ll y,
\]
so, we conclude that the integral over $H_1$ is bounded by
\[
\ll \frac{y}{\sqrt{|t|}}\exp\left(-\frac{y^{4/3}}{10|t|}\right) \ll \exp\left(-\frac{y^{4/3}}{20|t|}\right)
\]
since $|t| \leq C$ and $y$ is large.

For the $H_2$ integral, an identical bound holds and the derivation is the same.

\subsubsection*{$V_1$ and $V_2$ integrals}
First consider the integral over $V_1$. Applying the triangle inequality and Lemma \ref{integrandboundlemma}(ii) we get
\begin{align*}
\frac{1}{i\sqrt{\pi |t|}}\int_{V_1} \gamma(z) \exp\left(\frac{1}{|t|}(J_t(s)-z)^2 \right) & n^{-z}\, dz \\
&\ll \frac{1}{\sqrt{|t|}} \int_{V_1} \exp\left(\frac{1}{2|t|}(y-\Im z)^2\right)\, |dz| \\
&= \frac{1}{\sqrt{|t|}}\int_{-\infty}^{-y^{2/3}} \exp\left(-\frac{u^2}{2|t|}\right)\, du \\
&\ll \exp\left(-\frac{y^{4/3}}{20|t|}\right).
\end{align*}

A matching argument gives the same bound for the $V_2$ integral.

\subsubsection*{Conclusion of (i) and (ii)}
Adding together our estimates for all five integrals we get
\begin{equation} \label{eq:Btnwithadd}
B_{t,n}(s)=\gamma_t(s) \exp\left(-\frac{1}{4A}\log^2 n\right)n^{-s} (1+O(y^{-1/5})) + O\left(\exp\left(-\frac{y^{4/3}}{20|t|}\right)\right).
\end{equation}

In order to absorb the additive error term at the end into the multiplicative error term it suffices to show the following lower bound on the size of the main term: 
\begin{equation} \label{eq:mtlowerbd}
\left| \gamma_t(s) \exp\left(-\frac{1}{4A}\log^2 n\right)n^{-s} \right| \gg \exp\left(-\frac{2y^{6/5}}{|t|}\right).
\end{equation}
(This is sufficient because the exponent $6/5$ is less than $4/3$.)

To derive this lower bound, we will bound each of the factors individually. For the first factor, recall that $\gamma_t(s) \coloneqq \gamma(s) \exp\left(\frac{1}{|t|}(s-J_t(s))^2\right)$. Since $x \ll y^{1/4}$, we can apply Lemma \ref{gammasizelemma} to get
\[
|\gamma(s)| \gg \exp(-K' y)
\]
for some $K' > 0$, and furthermore
\begin{align*}
\left|\exp\left(\frac{1}{|t|}(s-J_t(s))^2\right)\right| &= \left|\exp\left(\frac{1}{|t|}\Re\left(\left(s-J_t(s)\right)^2\right)\right)\right| \\
&\geq \left|\exp\left(-\frac{1}{|t|}\left(\Im \left(s-J_t(s)\right)\right)^2\right)\right| \\
&\geq \exp(-O(1))
\end{align*}
where the last inequality comes from the fact that $\Im(s-J_t(s)) = O(|t|)$ by definition \eqref{eq:Jtdef} of $J_t(s)$. Combining these lower bounds gives
\begin{equation} \label{eq:gtlowerbound}
\left|\gamma_t(s)\right|\gg \exp(-K' y)
\end{equation}
for all sufficiently large $y$.

To lower bound the other two factors of \eqref{eq:mtlowerbd}, note $\Re\left(\frac{1}{A}\right)\leq 2 |t|$ for all large $y$ by \eqref{eq:Aest}, and $\log n \leq y^{3/5}/|t|$ by assumption, so
\[
\left| \exp\left(-\frac{1}{4A}\log^2 n\right) \right| \geq \exp\left(-\frac{|t|}{2}\log^2 n\right) 
\geq \exp\left(-\frac{y^{6/5}}{2|t|}\right)
\]
and also
\[
|n^{-s}| = \exp(-x \log n) \geq \exp\left(-\frac{y^{6/5}}{2|t|}\right)
\]
since $x\ll y^{1/4}$.

Hence, \eqref{eq:mtlowerbd} is true, so we deduce that
\begin{equation} \label{eq:Btnnoadd}
B_{t,n}(s)=\gamma_t(s) \exp\left(-\frac{1}{4A}\log^2 n\right)n^{-s} (1+O(y^{-1/5})).
\end{equation}
From this estimate we shall now deduce (i) and (ii) of the lemma. By taking absolute values in \eqref{eq:Btnnoadd} and applying $\Re\left(\frac{1}{A}\right) \geq \frac{|t|}{2}$, which follows from \eqref{eq:Aest}, we deduce (ii) immediately. To get (i), note that since $\log n \leq y^{1/3}/|t|$ in this case
\begin{align*} 
-\frac{1}{4A}\log^2 n &= -\frac{|t|}{4} \left(1+O\left(\frac{|t|}{y}\right)\right) \log^2 n \\
&= -\frac{|t|}{4}\log^2 n + O(y^{-1/3}),
\end{align*}
where we have used \eqref{eq:Aest} once again in the first line. Inserting this estimate into \eqref{eq:Btnnoadd} and using $\exp(O(y^{-1/3}))=1+O(y^{-1/3})$ gives (i).

\subsubsection*{Proof of (iii)}
Recall that the integral defining $B_{t,n}(s)$ is 
\[
B_{t,n}(s)=\frac{1}{i \sqrt{\pi |t|}} \int_{2-\infty i}^{2+\infty i} \gamma(z) e^{\frac{1}{|t|}(J_t(s)-z)^2} n^{-z}\, dz.
\]
We now work under the assumption that $n$ is large enough that $|t| \log n > y^{3/5}$.

We can rewrite the exponentials in the integrand by completing the square
\begin{equation} \label{eq:completesquare}
e^{\frac{1}{|t|}(J_t(s)-z)^2}n^{-z} = \exp\left(-\left(J_t(s)+\frac{|t|}{4}\log n\right)\log n\right) e^{\frac{1}{|t|}\left(z-\left(J_t(s)+\frac{|t|}{2}\log n\right)\right)^2},
\end{equation}
so if we choose to  integrate over the vertical line $\Re z = \Re \left(J_t(s)+\frac{|t|}{2}\log n\right)$, we see that there would be no oscillation coming these exponentials. Let $L$ denote this vertical line, and we now claim that $L$ lies in the right half plane. Indeed, because $\Re J_t(s) = x+O(\log y) = O(y^{1/4})$, and  $|t| \log n > y^{3/5}$, we see that
\begin{equation} \label{eq:reJtbound}
|\Re J_t(s)| \leq  \frac{|t|}{10}\log n
\end{equation}
for $y$ large enough, and so in particular $\Re \left(J_t(s)+\frac{|t|}{2}\log n\right) > 0$.

In the right half plane the integrand has no poles and decays rapidly in any vertical strip, so by Cauchy's theorem
\begin{equation} \label{eq:BtnL}
B_{t,n}(s)=\frac{1}{i \sqrt{\pi |t|}} \int_L \gamma(z) e^{\frac{1}{|t|}(J_t(s)-z)^2} n^{-z}\, dz.
\end{equation}

For all $z \in L$, note that \eqref{eq:reJtbound} implies $\Re z \asymp |t| \log n$ so we can apply Lemma~\ref{gboundlemma} to bound $\gamma(z)$, which gives $\gamma(z) \ll \exp(K'' |t| \log^{1.1} n)$ for some $K'' > 0$. Applying this bound and \eqref{eq:completesquare} to \eqref{eq:BtnL}, we can bound $B_{t,n}(s)$ by,
\begin{align*}
&\ll \frac{\exp(K'' |t| \log^{1.1} n)}{\sqrt{|t|}} \int_L \left|e^{\frac{1}{|t|}(J_t(s)-z)^2} n^{-z}\right|\, |dz| \\
&= \exp\left(-\Re \left(J_t(s)+\frac{|t|}{4}\log n\right)\log n\right)\frac{\exp(K'' |t| \log^{1.1} n)}{\sqrt{|t|}} \int_{-\infty}^{\infty} e^{-u^2/|t|}\, du \\
&\ll \exp\left(-\Re \left(J_t(s)+\frac{|t|}{4}\log n\right)\log n\right)\exp(K''|t| \log^{1.1} n).
\end{align*}
From \eqref{eq:reJtbound} it follows that $\Re \left(J_t(s)+\frac{|t|}{4}\log n\right) \geq \frac{|t|}{8} \log n$, so
\[
B_{t,n}(s) \ll \exp\left(-\frac{|t|}{8} \log^2 n\right) \exp(K'|t| \log^{1.1} n) \ll \exp\left(-\frac{|t|}{10} \log^2 n\right)
\]
for all $y$ sufficiently large (because this implies $\log n$ is large). 
\end{proof}
\subsection{Deducing Theorem \ref{mainthm} from Lemma \ref{mainlemma}}
Recall that $\xi^F_t(J_t(s))= \sum_{n=1}^\infty a_n B_{t,n}(s)$. We will split this sum into small, medium, and large $n$ as designated by Lemma \ref{mainlemma}, and then we will approximate each of the terms using the lemma.

It will be notationally convenient to define,
\[
\widetilde{F_t}(x) \coloneqq \sum_{n=1}^{\infty} \exp\left(-\frac{|t|}{4}\log^2 n\right)\frac{|a_n|}{n^x}
\]
which is the sum one gets when applying the triangle inequality to $|F_t(s)|$.

For the small $n$ terms of the sum (i.e. those for which $|t| \log n \leq y^{1/3}$) the lemma gives
\begin{align*}
\sum_{\text{small $n$}} a_n B_{t,n}(s) &=  \gamma_t(s) \sum_{\text{small $n$}} \exp\left(-\frac{|t|}{4}\log^2 n\right) \frac{a_n}{n^s} (1+O(y^{-1/5})) \nonumber \\
 &=  \gamma_t(s)\left(\sum_{\text{small $n$}} \exp\left(-\frac{|t|}{4}\log^2 n\right) \frac{a_n}{n^s} + O(y^{-1/5}\widetilde{F_t}(x)) \right).
\end{align*}
Rewriting the sum in the main term as $F_t(s)$ minus a tail sum, we get
\begin{equation} \label{eq:smallnest}
=  \gamma_t(s)\left(F_t(s) +O\left(y^{-1/5} \widetilde{F_t}(x)+\sum_{\log n > y^{1/3}/|t|}\exp\left(-\frac{|t|}{4}\log^2 n\right) \frac{|a_n|}{n^x} \right)\right).
\end{equation}
For the medium $n$ terms (which satisfy $y^{1/3} < |t|\log n \leq y^{3/5}$) the lemma gives
\begin{align} \label{eq:mediumnest}
\sum_{\text{medium $n$}} a_n B_{t,n}(s) &=  \gamma_t(s) \times O\left(\sum_{\log n > y^{1/3}/|t|}\exp\left(-\frac{|t|}{8}\log^2 n\right) \frac{|a_n|}{n^x} \right).
\end{align}
Adding \eqref{eq:smallnest} and \eqref{eq:mediumnest}, we get an estimate for the sum over all small and medium $n$,
\begin{align*}
=  \gamma_t(s)\left(F_t(s) +O\left(y^{-1/5} \widetilde{F_t}(x)+\sum_{\log n > y^{1/3}/|t|}\exp\left(-\frac{|t|}{8}\log^2 n\right) \frac{|a_n|}{n^x} \right)\right).
\end{align*}

To bound the tail sum in the error term note that since $|t| \log n > y^{1/3}$ and $a_n=O(n^2)$, 
\begin{align*}
\exp\left(-\frac{|t|}{8}\log^2 n\right) \frac{|a_n|}{n^x} &\ll \left(n^{-y^{1/3}/8}\right) \left(n^{2-x} \right)\\
&\leq n^{-\frac{y^{1/3}}{10}}
\end{align*}
where the second inequality holds for all sufficiently large $y$ because $x = O(y^{1/4})$. Summing this bound over $n > \exp(y^{1/3}/|t|)$ and applying the bound $\sum_{n > R} n^{-M} \leq 2 R^{1-M}$ (which holds for all $R,M \geq 2$, say) we conclude that
\begin{equation} \label{eq:smediumnest}
\sum_{\log n < \frac{y^{3/5}}{|t|}} a_n B_{t,n}(s) =  \gamma_t(s)\left(F_t(s) +O\left(y^{-1/5} \widetilde{F_t}(x)+\exp\left(-\frac{y^{2/3}}{15|t|}\right)\right)\right).
\end{equation}
for all sufficiently large $y$.

Lastly, for large $n$ (which satisfy $|t|\log n > y^{3/5}$) the lemma tells us
\[
\sum_{\text{large $n$}} a_n B_{t,n}(s) = O\left(\sum_{\log n > \frac{y^{3/5}}{|t|}} \exp\left(-\frac{|t|}{10}\log^2 n\right) |a_n|\right),
\]
and by using the same method that was just used to bound the other tail sum, we deduce the bound
\[
\sum_{\text{large $n$}} a_n B_{t,n}(s) =O\left(\exp\left(-\frac{y^{6/5}}{15|t|}\right)\right).
\]
In order to add this into the small and medium $n$ estimate \eqref{eq:smediumnest} and absorb this error term into the existing error terms we first need to divide through by the $\gamma_t(s)$ factor. Since we have seen previously in \eqref{eq:gtlowerbound} that $\gamma_t(s)$ only decays exponentially in $y$, we can write the large $n$ sum bound as
\[
\sum_{\text{large $n$}} a_n B_{t,n}(s) = \gamma_t(s) \times O\left(\exp\left(-\frac{y^{6/5}}{20|t|}\right)\right).
\]
Adding together all of the small, medium, and large $n$ then gives
\begin{equation} \label{eq:allnest}
\xi^F_t(J_t(s)) =  \gamma_t(s)\left(F_t(s) +O\left(y^{-1/5} \widetilde{F_t}(x)+\exp\left(-\frac{y^{2/3}}{20|t|}\right)\right)\right).
\end{equation}

We conclude the proof by bounding $\widetilde{F_t}(x)$ to simplify the error term. First we consider the case when $x \leq -2$. By applying $a_n=O(n^2)$, we see that
\[
\widetilde{F_t}(x) \ll \sum_{n=1}^\infty \exp\left(-\frac{|t|}{4}\log^2 n+2|x|\log n\right).
\]
Truncating this sum to only those values of $n$ satisfying $\frac{|t|}{4}\log n \leq 3|x|$ will incur only an $O(1)$ error because this condition implies that all the terms in the tail are $O(n^{-2})$, so we have
\begin{align*}
\widetilde{F_t}(x) \ll \sum_{n \leq \exp(12|x|/|t|)} \exp\left(-\frac{|t|}{4}\log^2 n+2|x|\log n\right)+O(1).
\end{align*}
It is easy to verify that the maximum possible value of a term in the sum is $\exp\left(\frac{4|x|^2}{|t|}\right)$, and so bounding each term by this quantity gives
\[
\widetilde{F_t}(x) \ll \exp\left(\frac{4|x|^2}{|t|}\right)\exp\left(\frac{12|x|}{|t|}\right)+O(1) \ll \exp\left(\frac{10|x|^2}{|t|}\right),
\]
where the second inequality holds because $|t| \leq C$ and because we are in the case $x \leq -2$.

For $x>-2$, one can prove a bound which decays exponentially in $x$ by similar methods to what we used above, but in the interest of keeping the error term simple we instead choose to use the trivial bound coming from the monotonicity of $\widetilde{F_t}(x)$. Since $\widetilde{F_t}(x)$ is decreasing in $x$, we naturally have
\begin{equation} \label{eq:Ftbound}
\widetilde{F_t}(x) \ll \exp\left(\frac{10}{|t|}\min(x,-2)^2\right).
\end{equation}
Inserting this into \eqref{eq:allnest} gives the theorem.

\section{Proof of Lemmas} \label{lemmasection}
\begin{replemma}{gammasizelemma}
Let $\gamma(s)$ be one of the functions described in condition (iii) of the extended Selberg class definition. Let $D>0$ and $0 \leq \theta < 1$, and let $s$ be a complex number which is at least unit distance away from the poles and zeros of $\gamma$ and which satisfies $|\Re s| \leq D |\Im s|^\theta$. 
There exist $K,K' >0$ (depending on $\gamma, D,$ and $\theta$) such that
\[
\exp(-K' |\Im s|) \ll |\gamma(s)| \ll \exp(-K |\Im s|),
\]
where the implicit constants may depend on $\gamma, D,$ and $\theta$.
\end{replemma}
\begin{proof}
Recall that 
\[
\gamma(s) \coloneqq \alpha s^m (s-1)^m Q^s \prod_{i=1}^k \Gamma(\omega_i s + \mu_i).
\]

We may assume $|\Im s|$ is large, because the small $|\Im s|$ case follows by compactness. (Indeed, since $s$ is assumed to be bounded away from the poles and zeros of $\gamma$ and since $|\Re s| \leq D |\Im s|^\theta$, one sees that $|\gamma(s)| \asymp 1$ in the $|\Im s| = O(1)$ case.)

Note that the polynomial and exponential factors in $\gamma(s)$ are insignificant because if $|\Im s|$ is large and $|\Re s| \leq D |\Im s|^\theta$, then
\[
\exp\left(-|\Im s|^{\theta'}\right) \leq |s^m (s-1)^m Q^s| \leq \exp\left(|\Im s|^{\theta'}\right)
\]
for some $\theta < \theta' < 1$. Hence, it is enough to show that each $\Gamma$ factor obeys
\begin{equation}\label{eq:stirlingdecay}
\exp(-K'_i |\Im s|) \leq |\Gamma(\omega_i s + \mu_i)| \leq \exp(-K_i |\Im s|)
\end{equation}
for $K'_i, K_i>0$.

Fix some $i$ and let $z \coloneqq \omega_i s+\mu_i$. Since $\omega_i \in \R$ and $|\Im s|$ is large, we have $|\Im z| \asymp |\Im s|$ and $|\Re z| \ll |\Im z|^\theta$, and we may assume $|\Im z|$ is large as well. 

Stirling's approximation gives
\begin{align*}
|\Gamma(z)| &= \exp\left(\Re\left(z \Log z - z + \frac{1}{2}\Log\frac{2\pi}{z}\right)+O(1/|z|)\right) \\
&= \exp\left(-(\Im z) \Arg z + O\left(|\Im z|^{\theta'}\right)\right).
\end{align*}

Since $|\Re z| \ll |\Im z|^\theta$ and $|\Im z|$ is large, we deduce that $\pi/4 \leq \Arg z \leq 3\pi/4$ if $\Im z > 0$ and $-3\pi/4 \leq \Arg z \leq -\pi/4$ if $\Im z < 0$. Hence, we can conclude
\[
\exp(-C'|\Im z|) \leq |\Gamma(z)| \leq \exp(-C |\Im z|),
\]
and then \eqref{eq:stirlingdecay} follows since $|\Im z| \asymp |\Im s|$.
\end{proof}

\begin{replemma}{Rlemma}
Let
\[
h(x) \coloneqq \frac{1}{2\pi i} \int_{1-\infty i}^{1+\infty i} \prod_{j=1}^{k} \Gamma(a_j w + b_j) x^{-w}\, dw
\]
for some $a_j > 0$ and $b_j\in \C$ with $\Re b_j \geq 0$. Then there exists a $\delta>0$ (depending on the $a_j$ and $b_j$ values) such that $h(x) \ll e^{-x^{\delta}}$ for all $x \geq 1$.
\end{replemma}
\begin{proof}
Let $\mathcal{R}$ denote the set of functions $f \in C\left((0,\infty)\right)$ for which,
\begin{enumerate}[(i)]
\item there exists a $\delta>0$ such that $f(x) \ll e^{-x^\delta}$ for all $x \geq 1$, and
\item for every $\kappa > 0$, the bound $f(x) \ll_\kappa x^{-\kappa}$ holds for all $x>0$.\end{enumerate}
To prove the statement it clearly suffices to show that $h \in \mathcal{R}$. In order to do this we first make several observations about $\mathcal{R}$ and about Mellin transforms.

Note that for any function $f \in\mathcal{R}$, the bound (ii) implies that the Mellin transform
\[
F(w) \coloneqq \int_0^\infty f(x) x^w\, \frac{dx}{x}
\]
of $f$ is defined for all $\Re w > 0$. Moreover, if $F$ is integrable over some vertical line $\Re w = c > 0$, then $f$ can be recovered from $F$ via the inverse Mellin transform,
\[
f(x) = \frac{1}{2\pi i} \int_{c-\infty i}^{c+\infty i} F(w) x^{-w}\, dw.
\]
This follows directly from the Fourier inversion formula for $L^1$ functions if one performs the necessary changes of variables.

If $f$ and $g$ are functions in $\mathcal{R}$, we define their multiplicative convolution to be
\begin{equation} \label{eq:convdef}
f \star g(x) \coloneqq \int_{0}^\infty f\left(\frac{x}{y}\right)g(y)\, \frac{dy}{y},
\end{equation}
and we claim that $\mathcal{R}$ is closed under this operation. Note that if $f,g\in\mathcal{R}$, then $f \star g$ is continuous by a standard application of the dominated convergence theorem, so it suffices to show that this function satisfies the bounds (i) and (ii). 

Let $\delta>0$ be small enough so that both $f(y) \ll e^{-y^\delta}$ and $g(y) \ll e^{-y^\delta}$ for all $y\geq 1$. For $x\geq 1$, we bound the $f \star g(x)$ integral by splitting it into two parts,
\begin{align*}
f \star g(x) &= \left(\int_0^{x^{1/2}}+\int_{x^{1/2}}^\infty\right)f\left(\frac{x}{y}\right)g(y)\,\frac{dy}{y},
\end{align*}
which we denote by I and II respectively.

To bound II, apply $f\left(\frac{x}{y}\right) \ll \left(\frac{x}{y}\right)^{-\delta}$ and $g(y) \ll e^{-y^\delta}$ to get
\[
\text{II} \ll \int_{x^{1/2}}^\infty \left(\frac{x}{y}\right)^{-\delta} e^{-y^\delta}\, \frac{dy}{y} 
= x^{-\delta} \int_{x^{1/2}}^\infty y^{\delta-1} e^{-y^\delta}\, dy
\ll e^{-x^{\delta/2}}.
\]
To bound I, note that the substitution $y' \coloneqq \frac{x}{y}$ turns this integral into an integral which is identical to II but with the roles of $f$ and $g$ reversed. Hence the same bound holds, and we can conclude that
\[
f\star g(x) \ll e^{-x^{\delta/2}}
\]
for all $x \geq 1$. 

It remains to show that $f\star g(x) \ll_\kappa x^{-\kappa}$ for all $x>0$ and  $\kappa>0$. Fixing some $\kappa > 0$, it suffices to prove this bound for $0< x <1$ because we have already proved a superior bound when $x \geq 1$. To obtain the bound when $x<1$, we split the $f\star g(x)$ integral into three parts
\begin{align*}
f \star g(x) = \left(\int_0^x + \int_x^1+\int_1^\infty\right)f\left(\frac{x}{y}\right)g(y)\,\frac{dy}{y}
\end{align*}
which we denote by I, II, and III.

For III, the bounds $f\left(\frac{x}{y}\right) \ll \left(\frac{x}{y}\right)^{-\kappa}$ and $g(y) \ll e^{-y^\delta}$ hold, so
\[
\text{III} \ll \int_1^\infty \left(\frac{x}{y}\right)^{-\kappa} e^{-y^\delta}\, \frac{dy}{y} \ll x^{-\kappa}.
\]
For II, the bounds $f\left(\frac{x}{y}\right) \ll \left(\frac{x}{y}\right)^{-\kappa/2}$ and $g(y) \ll y^{-\kappa/2}$ hold, so
\[
\text{II} \ll \int_x^1 \left(\frac{x}{y}\right)^{-\kappa/2} y^{-\kappa/2}\, \frac{dy}{y}
= x^{-\kappa/2} \log{\frac{1}{x}} \ \ll x^{-\kappa}.
\]
For I, we can again apply the substitution $y' \coloneqq \frac{x}{y}$ to turn this integral into an identical integral to III but with the roles of $f$ and $g$ reversed. Hence,
\[
f\star g(x) \ll x^{-\kappa},
\]
so we can conclude that $f \star g \in \mathcal{R}$.

Now suppose that $f, g \in \mathcal{R}$ with $F,G$ their respective Mellin transforms. Using Fubini's theorem and a change of variables, it is straightforward to show that the Mellin transform of $f\star g$ is the product $FG$. Furthermore, if $F(w)G(w)$ is integrable over a line $\Re w = c >0$ then our previous remark about Mellin inversion implies that $f \star g$ is the inverse Mellin transform of $FG$. We may now apply this result inductively to prove the lemma.

Using the fact that the inverse Mellin transform of $\Gamma(w)$ is $e^{-x}$, one can verify that the inverse Mellin transform of $\Gamma(a_j w+b_j)$ is $\frac{1}{a_j}x^{b_j/a_j} e^{-x^{1/a_j}}$, which is a member of $\mathcal{R}$. Any product of these gamma functions will be integrable over the line $\Re w=2$ because of the exponential decay of the gamma function on vertical lines. Hence, the inverse Mellin transform
\[
h(x) = \frac{1}{2\pi i} \int_{2-\infty i}^{2+\infty i} \prod_{j=1}^{k} \Gamma(a_j w + b_j) x^{-w}\, dw
\]
is a multiplicative convolution of $k$ functions each of which are in $\mathcal{R}$, so $h$ is in $\mathcal{R}$ as well.

\end{proof}

\begin{replemma}{Ftzerolemma}
For any $t < 0$, $F_t$ has a zero.
\end{replemma}
\begin{proof}
Fix $t<0$, and recall that $F_t(s) \coloneqq \sum_{n=1}^{\infty} \exp\left(-\frac{|t|}{4}\log^2 n\right)\frac{a_n}{n^{s}}$ is an entire function. Furthermore, the bound \eqref{eq:Ftbound} on the function $\widetilde{F_t}(x) \coloneqq \sum_{n=1}^{\infty} \exp\left(-\frac{|t|}{4}\log^2 n\right)\frac{|a_n|}{n^{x}}$ implies
\[
|F_t(x+iy)| \leq \widetilde{F_t}(x) \ll \exp\left(\frac{10}{|t|}\min(x,-2)^2\right),
\]
so $F_t$ is of order at most two.

Now suppose for the sake of contradiction that $F_t$ has no zeros. Then by the Hadamard factorization theorem, $F_t(s)=\exp(P(s))$ where $P$ is a polynomial of degree at most two. Using the fact that $F_t(s)$ is uniformly bounded in all half-planes $\{\Re s > c\}$, one may verify that the only possible choices for $P(s)$ are polynomials of the form $P(s) = -\lambda s + \rho$  where $\lambda \geq 0$, $\rho \in \C$. Hence,
\[
\sum_{n=1}^{\infty} \exp\left(-\frac{|t|}{4}\log^2 n\right)\frac{a_n}{n^{s}} = e^{\rho} \exp(-\lambda s)
\]
for all $s\in\C$. Now note that the right-hand side of this equality is a generalized Dirichlet series with only one term whereas the left-hand side is a Dirichlet series with multiple terms (since $a_n$ is nonzero for more than one $n$ as we have noted previously). By the uniqueness of coefficients of generalized Dirichlet series such an equality is impossible, so we get a contradiction.
\end{proof}

\begin{replemma}{gapproxlemma}
For any $\eps > 0$, define the region $S_\eps \coloneqq \left\{|\Arg w| < \pi - \eps, |w|>\eps\right\}$. Let $z,z_0 \in S_\eps$ and $|z-z_0| \leq D |z_0|^{2/3}$ for some $D > 0$. Then for all $z, z_0$ that are at least unit distance away from the poles and zeros of $\gamma$, we have the estimate
\begin{multline*}
\gamma(z) = \gamma(z_0) \exp\left(\left(\log Q+\sum_{j=1}^{k} \omega_j \Log(\omega_j z_0)\right)(z-z_0)+\frac{\sum_{j=1}^k \omega_j}{2z_0} (z-z_0)^2\right) \\
\times \left(1+O\left(\frac{1+|z-z_0|}{|z_0|}+\frac{|z-z_0|^3}{|z_0|^2}\right)\right)
\end{multline*}
where the implicit constant may depend on $\eps$ and $D$.
\end{replemma}
\begin{proof}
We may assume that $|z_0|$ is large, as the small $|z_0|$ case holds by compactness. We allow all implicit constants to depend on $\eps$ and $D$.

Recall that $\gamma(z) \coloneqq \alpha z^m (z-1)^m Q^z \prod_{j=1}^{k} \Gamma(\omega_i z + \mu_i)$. We will consider the factors of $\gamma$ separately and show how each one differs when evaluating at $z$ versus $z_0$. 

To handle the polynomial factor $\alpha z^m (z-1)^m$, note that for all $|z_0|$ large we have
\[
z = z_0\left(1+ O\left(\frac{|z-z_0|}{|z_0|}\right)\right) \quad \text{ and } \quad z - 1= (z_0-1)\left(1+ O\left(\frac{|z-z_0|}{|z_0|}\right)\right),
\]
hence
\begin{equation} \label{eq:polyapprox}
\alpha z^m(z-1)^m = \alpha z_0^m (z_0-1)^m \left(1+O\left(\frac{|z-z_0|}{|z_0|}\right)\right).
\end{equation}

For the exponential term in $\gamma$, there is a trivial equality
\begin{equation} \label{eq:expapprox}
Q^z = Q^{z_0} \exp((\log Q)(z-z_0)).
\end{equation}

To get estimates for the $\Gamma$ factors in $\gamma$ we recall Stirling's formula which states that for any $w\in S_\eps$ with $|w| \geq 1$,
\begin{equation} \label{eq:stirlings}
\Gamma(w) = \sqrt{2\pi }\exp\left(w \Log w - w -\frac{1}{2} \Log w\right)\left(1 + O\left(\frac{1}{|w|}\right)\right).
\end{equation}

Suppose $w,w_0 \in S_\eps$ and $|w-w_0| \ll |w_0|^{2/3}$. By assuming $|w_0|$ is sufficiently large we have $\left|\Arg\left(1+\frac{w-w_0}{w_0}\right)\right| < \eps$, and we then deduce
\begin{equation} \label{eq:logid}
\Log w = \Log w_0 + \Log\left(1+\frac{w-w_0}{w_0}\right).
\end{equation}

By applying Stirling's formula at $w$ and $w_0$ (and assuming $|w_0|$ is large), then using \eqref{eq:logid} and the fact that $|w|\asymp |w_0|$ gives
\begin{align*}
\frac{\Gamma(w)}{\Gamma(w_0)} &= \exp\left(w \Log w - w - w_0 \Log w_0 + w_0 - \frac{1}{2} \Log\left(1+\frac{w-w_0}{w_0}\right) \right) \\
& \hspace{20em} \times \left(1+O\left(\frac{1}{|w_0|}\right)\right) \\
&= \exp\left(w \Log w - w - w_0 \Log w_0 + w_0 \right)\left(1+O\left(\frac{1+|w-w_0|}{|w_0|}\right)\right)
\end{align*}
where the second line follows from the first because $\Log(1+s) = O(|s|)$ for $|s| \leq \frac{1}{2}$, and $\exp(s) = 1 + O(|s|)$ for $s$ bounded.

If we consider just the expression inside the exponential, one can show that
\begin{multline*}
w \Log w - w - w_0 \Log w_0 + w_0 = \Log(w_0) (w-w_0) + \frac{1}{2 w_0} (w-w_0)^2 \\ + O\left(\frac{|w-w_0|^3}{|w_0|^2}\right)
\end{multline*}
by applying \eqref{eq:logid} to the logarithm, then applying the Taylor approximation $\Log(1+s) = s-\frac{s^2}{2}+O(s^3)$ which holds for all $|s| \leq \frac{1}{2}$, and gathering together error terms. 

Plugging this back into our expression for $\frac{\Gamma(w)}{\Gamma(w_0)}$ and then applying $\exp(1+s)=1+O(|s|)$ for $s$ bounded again,
\begin{multline} \label{eq:gammaratio}
\frac{\Gamma(w)}{\Gamma(w_0)} = \exp\left(\Log(w_0) (w-w_0) + \frac{1}{2 w_0} (w-w_0)^2\right) \\\times \left(1+O\left(\frac{1+|w-w_0|}{|w_0|}+\frac{|w-w_0|^3}{|w_0|^2}\right)\right).
\end{multline}

Taking $w \coloneqq \omega_i z + \mu_i$ and $w_0 \coloneqq \omega_i z_0 + \mu_i$ note that from the assumptions that $z,z_0 \in S_\eps$, $|z-z_0| \leq D |z_0|^{2/3}$, and $|z_0|$ is large, we get the corresponding facts that $w,w_0 \in S_{\eps'}$, $|w-w_0| \ll |w_0|^{2/3}$, and $|w_0|$ is large. Hence, by \eqref{eq:gammaratio}
\begin{multline} \label{eq:gammaratio2}
\frac{\Gamma(\omega_i z + \mu_i)}{\Gamma(\omega_i z_0 + \mu_i)} = \exp\left(\omega_i \Log(\omega_i z_0 + \mu_i) (z-z_0) + \frac{\omega_i^2}{2( \omega_i z_0 + \mu_i)} (z-z_0)^2\right) \\
\times \left(1+O\left(\frac{1+|z-z_0|}{|z_0|}+\frac{|z-z_0|^3}{|z_0|^2}\right)\right).
\end{multline}

When $|z_0|$ is large,
\[
\Log(\omega_i z_0 + \mu_i) = \Log(\omega_i z_0) + O\left(\frac{1}{|z_0|}\right)
\]
and
\[
\frac{1}{\omega_i z_0 + \mu_i} = \frac{1}{\omega_i z_0} + O\left(\frac{1}{|z_0|^2}\right).
\]
Inserting these into \eqref{eq:gammaratio2} gives
\begin{multline} \label{eq:gammaratio3}
\frac{\Gamma(\omega_i z + \mu_i)}{\Gamma(\omega_i z_0 + \mu_i)} = \exp\left(\omega_i \Log(\omega_i z_0) (z-z_0) + \frac{\omega_i}{2z_0} (z-z_0)^2\right) \\
\times \left(1+O\left(\frac{1+|z-z_0|}{|z_0|}+\frac{|z-z_0|^3}{|z_0|^2}\right)\right).
\end{multline}

Applying \eqref{eq:gammaratio3} to each $\Gamma$ factor in $\gamma$ and combining this with \eqref{eq:polyapprox} and \eqref{eq:expapprox} gives the lemma.

\end{proof}

\begin{replemma}{integrandboundlemma}
Let $t$, $s=x+iy$, and $C$ satisfy the same hypotheses as in Theorem \ref{mainthm}. Let $n \leq \exp(y^{3/5}/|t|)$ (i.e. the small/medium case of Lemma~\ref{mainlemma}) and let
\[
\mathcal{I}(z) \coloneqq \gamma(z)e^{\frac{1}{|t|}(J_t(s)-z)^2}n^{-z}
\]
be the integrand of \eqref{eq:Btndef}. Then for all $y$ sufficiently large (depending on $C$) the following estimates hold (where the implicit constants may also depend on $C$):
\begin{enumerate}[(i)]
\item If $|x-\Re z| \leq y^{3/5}$ and $\frac{1}{2} y^{2/3} \leq |y-\Im z| \leq 2y^{2/3}$ then 
\[
\mathcal{I}(z) \ll \exp\left(-\frac{y^{4/3}}{10|t|}\right).
\]
\item If $\Re z = 2$ and $|y-\Im z| \geq y^{2/3}$ then
\[
\mathcal{I}(z) \ll \exp\left(-\frac{1}{2|t|}(y-\Im z)^{2}\right).
\]
\end{enumerate}
\end{replemma}
\begin{proof}
We shall bound the three factors of $\mathcal{I}(z)$ individually and prove (i) and (ii) in parallel. We allow all implicit constants to depend on $C$, and we freely assume that $y$ large enough to make various assumptions hold.

From the definition \eqref{eq:Jtdef} of $J_t(s)$, and the assumptions that $x =O(y^{1/4})$ and $t =O(1)$ it is clear that
\[
\Im J_t(s) = y + O(1) \quad\text{ and }\quad \Re J_t(s) = x+O(\log y),
\]
so if $z$ satisfies the hypotheses in either (i) or (ii), then
\begin{align*}
\Re\left( ( J_t(s)-z)^2\right) &= (x+O(\log y)-\Re z)^2-(y+O(1)-\Im z)^2 \\
&\leq -\frac{1}{2}(y-\Im z)^2.
\end{align*}
for all $y$ sufficiently large (where we have used the assumption that $x \ll y^{1/4}$ again in the (ii) case). Hence we can bound the first exponential factor in $\mathcal{I}(z)$ by
\begin{equation} \label{eq:expJtbound}
\left| e^{\frac{1}{|t|}(J_t(s)-z)^2} \right | \leq  \exp\left(-\frac{1}{2|t|}(y-\Im z)^{2}\right).
\end{equation}

To handle the $\gamma(z)$ factor we can apply Lemma \ref{gammasizelemma}. Note that the hypotheses in (i) imply
\[
|\Re z| \leq |x|+y^{3/5} \ll y^{3/5} \quad\text{ and }\quad \Im z \asymp y
\]
and the in case (ii) we have $\Re z = 2$, so in both cases Lemma \ref{gammasizelemma} implies that $\gamma(z) = O(1)$.

In (ii) it is clear that the $n^{-z}=O(1)$, and in (i) note that  $\Re z \ll y^{3/5}$ and $\log n \leq y^{3/5}/|t|$, so
\[
n^{-z} = \exp(-(\Re z) \log n) \leq \exp\left(\frac{K y^{6/5}}{|t|}\right).
\]

Hence, to get the conclusion in (i), we put together the above bounds and the fact that $|y-\Im z| \geq \frac{1}{2}y^{2/3}$ which gives
\begin{align*}
\mathcal{I}(z) &\ll \exp\left(-\frac{1}{2|t|}(y-\Im z)^{2}\right)\exp\left(\frac{K y^{6/5}}{|t|}\right) \\
&\leq \exp\left(-\frac{y^{4/3}}{8|t|}\right)\exp\left(\frac{K y^{6/5}}{|t|}\right) \\
&\leq  \exp\left(-\frac{y^{4/3}}{10|t|}\right) 
\end{align*}
for all $y$ sufficiently large.

Similarly, the conclusion of (ii) is immediate from \eqref{eq:expJtbound} and the $O(1)$ bounds on the other two factors.

\end{proof}

\begin{replemma}{gboundlemma}
There is some $K>0$ such that $|\gamma(z)| \leq \exp(K (\Re z)^{1.1})$ uniformly for any $z$ with $\Re z \geq 1$.
\end{replemma}
\begin{proof}
Recall that 
\[
\gamma(z) \coloneqq \alpha z^m (z-1)^m Q^z \prod_{i=1}^k \Gamma(\omega_i z + \mu_i).
\]

Consider two cases where either $|\Im z| > (\Re z)^2$ or  $|\Im z| \leq (\Re z)^2$. In the first case, we see $\gamma(z) = O(1)$ by Lemma~\ref{gammasizelemma}. In the second case, one can bound each of the $\Gamma$ factors in $\gamma$ with the bound
\[
|\Gamma(\omega_i z + \mu_i)| \leq \Gamma(\Re(\omega_i z + \mu_i)) \leq \exp(K_0 (\Re z)^{1.1})
\]
for some $K_0 > 0$. Since, $|z| \ll (\Re z)^2$ and $|Q^z| = \exp((\log Q)(\Re z))$ the bound follows.
\end{proof}

\bibliographystyle{amsplain}
\bibliography{references}

@book{ahlfors,
	author = {Ahlfors, L. V.},
	edition = {Third},
	isbn = {0-07-000657-1},
	mrclass = {30-01},
	mrnumber = {510197},
	note = {An introduction to the theory of analytic functions of one complex variable, International Series in Pure and Applied Mathematics},
	pages = {xi+331},
	publisher = {McGraw-Hill Book Co., New York},
	title = {Complex analysis},
	url = {https://mathscinet.ams.org/mathscinet-getitem?mr=510197},
	year = {1978}}

@article{polymath2019,
	author = {Polymath, D. H. J.},
	doi = {10.1007/s40687-019-0193-1},
	fjournal = {Research in the Mathematical Sciences},
	issn = {2522-0144},
	journal = {Res. Math. Sci.},
	mrclass = {30D10 (11M26 30D15)},
	mrnumber = {4011563},
	number = {3},
	pages = {Paper No. 31, 67},
	title = {Effective approximation of heat flow evolution of the {R}iemann {$\xi$} function, and a new upper bound for the de {B}ruijn-{N}ewman constant},
	url = {https://mathscinet.ams.org/mathscinet-getitem?mr=4011563},
	volume = {6},
	year = {2019}}

@article{conrey1993g,
	author = {J.B. Conrey and A. Ghosh},
	doi = {10.1215/S0012-7094-93-07225-0},
	fjournal = {Duke Mathematical Journal},
	issn = {0012-7094},
	journal = {Duke Math. J.},
	mrclass = {11M41 (11M06)},
	mrnumber = {1253620},
	mrreviewer = {James Lee Hafner},
	number = {3},
	pages = {673--693},
	title = {On the {S}elberg class of {D}irichlet series: small degrees},
	url = {https://doi.org/10.1215/S0012-7094-93-07225-0},
	volume = {72},
	year = {1993}}

@article{kaczorowski1999p,
	author = {J. Kaczorowski and A. Perelli},
	doi = {10.1007/BF02392574},
	fjournal = {Acta Mathematica},
	issn = {0001-5962},
	journal = {Acta Math.},
	mrclass = {11M41 (11M06 11M36 33C90)},
	mrnumber = {1710182},
	mrreviewer = {James Lee Hafner},
	number = {2},
	pages = {207--241},
	title = {On the structure of the {S}elberg class. {I}. {$0\leq d\leq 1$}},
	url = {https://doi.org/10.1007/BF02392574},
	volume = {182},
	year = {1999}}

@incollection{kaczorowski2006,
	author = {J. Kaczorowski},
	booktitle = {Analytic number theory},
	doi = {10.1007/978-3-540-36364-4_4},
	mrclass = {11M36 (11M06)},
	mrnumber = {2277660},
	mrreviewer = {Giuseppe Molteni},
	pages = {133--209},
	publisher = {Springer, Berlin},
	series = {Lecture Notes in Math.},
	title = {Axiomatic theory of {$L$}-functions: the {S}elberg class},
	url = {https://doi.org/10.1007/978-3-540-36364-4_4},
	volume = {1891},
	year = {2006}}

@inproceedings{selberg1989,
	author = {A. Selberg},
	booktitle = {Proceedings of the {A}malfi {C}onference on {A}nalytic {N}umber {T}heory ({M}aiori, 1989)},
	mrclass = {11M41},
	mrnumber = {1220477},
	mrreviewer = {D. R. Heath-Brown},
	pages = {367--385},
	publisher = {Univ. Salerno, Salerno},
	title = {Old and new conjectures and results about a class of {D}irichlet series},
	year = {1992}}

@article{newman1976,
	author = {C. M. Newman},
	doi = {10.2307/2041319},
	fjournal = {Proceedings of the American Mathematical Society},
	issn = {0002-9939},
	journal = {Proc. Amer. Math. Soc.},
	mrclass = {10H05 (42A68)},
	mrnumber = {434982},
	mrreviewer = {T. M. Apostol},
	number = {2},
	pages = {245--251 (1977)},
	title = {Fourier transforms with only real zeros},
	url = {https://doi.org/10.2307/2041319},
	volume = {61},
	year = {1976}}

@book{titchmarsh1986,
	author = {E. C. Titchmarsh},
	edition = {Second},
	isbn = {0-19-853369-1},
	mrclass = {11M06},
	mrnumber = {882550},
	mrreviewer = {Matti Jutila},
	note = {Edited and with a preface by D. R. Heath-Brown},
	pages = {x+412},
	publisher = {The Clarendon Press, Oxford University Press, New York},
	title = {The theory of the {R}iemann zeta-function},
	year = {1986}}

@article{andrade2014cm,
	author = {J. Andrade and A. Chang and S. J. Miller},
	doi = {10.1016/j.jnt.2014.04.021},
	fjournal = {Journal of Number Theory},
	issn = {0022-314X},
	journal = {J. Number Theory},
	mrclass = {11M20 (11M26 11M50 11Y35 11Y60 14G10)},
	mrnumber = {3239152},
	mrreviewer = {Ayyadurai Sankaranarayanan},
	pages = {70--91},
	title = {Newman's conjecture in various settings},
	url = {https://doi.org/10.1016/j.jnt.2014.04.021},
	volume = {144},
	year = {2014}}

@article{stopple2014,
	author = {J. Stopple},
	doi = {10.7169/facm/2014.51.1.2},
	fjournal = {Uniwersytet im. Adama Mickiewicza w Poznaniu. Wydzia\l Matematyki i Informatyki. Functiones et Approximatio Commentarii Mathematici},
	issn = {0208-6573},
	journal = {Funct. Approx. Comment. Math.},
	mrclass = {11M20 (11M26 11M50 11Y35)},
	mrnumber = {3263068},
	mrreviewer = {Mehdi Hassani},
	number = {1},
	pages = {23--41},
	title = {Notes on low discriminants and the generalized {N}ewman conjecture},
	url = {https://doi.org/10.7169/facm/2014.51.1.2},
	volume = {51},
	year = {2014}}

@article{polya1927,
	author = {G. P\'{o}lya},
	doi = {10.1515/crll.1927.158.6},
	fjournal = {Journal f\"{u}r die Reine und Angewandte Mathematik. [Crelle's Journal]},
	issn = {0075-4102},
	journal = {J. Reine Angew. Math.},
	mrclass = {DML},
	mrnumber = {1581126},
	pages = {6--18},
	title = {\"{U}ber trigonometrische {I}ntegrale mit nur reellen {N}ullstellen},
	url = {https://doi.org/10.1515/crll.1927.158.6},
	volume = {158},
	year = {1927}}

@article{polya1926,
	author = {G. P\'{o}lya},
	doi = {10.1007/BF02565336},
	fjournal = {Acta Mathematica},
	issn = {0001-5962},
	journal = {Acta Math.},
	mrclass = {DML},
	mrnumber = {1555227},
	number = {3-4},
	pages = {305--317},
	title = {Bemerkung \"{U}ber die {I}ntegraldarstellung der {R}iemannschen {$\xi$}-{F}unktion},
	url = {https://doi.org/10.1007/BF02565336},
	volume = {48},
	year = {1926}}

@article{debruijn1950,
	author = {N. G. de Bruijn},
	fjournal = {Duke Mathematical Journal},
	issn = {0012-7094},
	journal = {Duke Math. J.},
	mrclass = {30.0X},
	mrnumber = {37351},
	mrreviewer = {J. Korevaar},
	pages = {197--226},
	title = {The roots of trigonometric integrals},
	url = {http://projecteuclid.org/euclid.dmj/1077476111},
	volume = {17},
	year = {1950}}

@article{bohr1922,
	author = {H. Bohr},
	doi = {10.1007/BF01449609},
	fjournal = {Mathematische Annalen},
	issn = {0025-5831},
	journal = {Math. Ann.},
	mrclass = {DML},
	mrnumber = {1512052},
	number = {1},
	pages = {115--122},
	title = {\"{U}ber eine quasi-periodische {E}igenschaft {D}irichletscher {R}eihen mit {A}nwendung auf die {D}irichletschen {$L$}-{F}unktionen},
	url = {https://doi.org/10.1007/BF01449609},
	volume = {85},
	year = {1922}}

@article{csordas1988nv,
	author = {G. Csordas and T. S. Norfolk and R. S. Varga},
	doi = {10.1007/BF01400887},
	fjournal = {Numerische Mathematik},
	issn = {0029-599X},
	journal = {Numer. Math.},
	mrclass = {30D10 (11Y35 65E05)},
	mrnumber = {945095},
	mrreviewer = {G. Opfer},
	number = {5},
	pages = {483--497},
	title = {A lower bound for the de {B}ruijn-{N}ewman constant {$\Lambda$}},
	url = {https://doi.org/10.1007/BF01400887},
	volume = {52},
	year = {1988}}

@article{csordas1991rv,
	author = {G. Csordas and A. Ruttan and R. S. Varga},
	doi = {10.1007/BF02142328},
	fjournal = {Numerical Algorithms},
	issn = {1017-1398},
	journal = {Numer. Algorithms},
	mrclass = {30D10 (30D15 65E05)},
	mrnumber = {1135299},
	mrreviewer = {Timothy S. Norfolk},
	number = {3},
	pages = {305--329},
	title = {The {L}aguerre inequalities with applications to a problem associated with the {R}iemann hypothesis},
	url = {https://doi.org/10.1007/BF02142328},
	volume = {1},
	year = {1991}}

@article{csordas1993osv,
	author = {G. Csordas and A. M. Odlyzko and W. Smith and R. S. Varga},
	fjournal = {Electronic Transactions on Numerical Analysis},
	journal = {Electron. Trans. Numer. Anal.},
	mrclass = {11M26 (30D10 30D15 65E05)},
	mrnumber = {1253639},
	mrreviewer = {Harvey Cohn},
	number = {Dec.},
	pages = {104--111 (electronic only)},
	title = {A new {L}ehmer pair of zeros and a new lower bound for the de {B}ruijn-{N}ewman constant {$\Lambda$}},
	volume = {1},
	year = {1993}}

@article{csordas1994sv,
	author = {G. Csordas and W. Smith and R. S. Varga},
	doi = {10.1007/BF01205170},
	fjournal = {Constructive Approximation. An International Journal for Approximations and Expansions},
	issn = {0176-4276},
	journal = {Constr. Approx.},
	mrclass = {30D10 (30D15)},
	mrnumber = {1260363},
	mrreviewer = {Timothy S. Norfolk},
	number = {1},
	pages = {107--129},
	title = {Lehmer pairs of zeros, the de {B}ruijn-{N}ewman constant {$\Lambda$}, and the {R}iemann hypothesis},
	url = {https://doi.org/10.1007/BF01205170},
	volume = {10},
	year = {1994}}

@incollection{norfolk1992rv,
	author = {T. S. Norfolk and A. Ruttan and R. S. Varga},
	booktitle = {Progress in approximation theory ({T}ampa, {FL}, 1990)},
	doi = {10.1007/978-1-4612-2966-7_17},
	mrclass = {30D10 (11Y35)},
	mrnumber = {1240792},
	mrreviewer = {G. Opfer},
	pages = {403--418},
	publisher = {Springer, New York},
	series = {Springer Ser. Comput. Math.},
	title = {A lower bound for the de {B}ruijn-{N}ewman constant {$\Lambda$}. {II}},
	url = {https://doi.org/10.1007/978-1-4612-2966-7_17},
	volume = {19},
	year = {1992}}

@incollection{odlyzko2000,
	author = {A. M. Odlyzko},
	doi = {10.1023/A:1016677511798},
	fjournal = {Numerical Algorithms},
	issn = {1017-1398},
	journal = {Numer. Algorithms},
	mrclass = {30D15 (30D20)},
	mrnumber = {1827160},
	note = {Mathematical journey through analysis, matrix theory and scientific computation (Kent, OH, 1999)},
	number = {1-4},
	pages = {293--303},
	title = {An improved bound for the de {B}ruijn-{N}ewman constant},
	url = {https://doi.org/10.1023/A:1016677511798},
	volume = {25},
	year = {2000}}

@article{saouter2011gd,
	author = {Y. Saouter and X. Gourdon and P. Demichel},
	doi = {10.1090/S0025-5718-2011-02472-5},
	fjournal = {Mathematics of Computation},
	issn = {0025-5718},
	journal = {Math. Comp.},
	mrclass = {11M26 (11-04 11Y35 11Y60)},
	mrnumber = {2813360},
	number = {276},
	pages = {2281--2287},
	title = {An improved lower bound for the de {B}ruijn-{N}ewman constant},
	url = {https://doi.org/10.1090/S0025-5718-2011-02472-5},
	volume = {80},
	year = {2011}}

@article{teriele1991,
	author = {H. J. J. te Riele},
	doi = {10.1007/BF01385647},
	fjournal = {Numerische Mathematik},
	issn = {0029-599X},
	journal = {Numer. Math.},
	mrclass = {30D10 (11M26 30D15 65E05)},
	mrnumber = {1083527},
	mrreviewer = {George L. Csordas},
	number = {6},
	pages = {661--667},
	title = {A new lower bound for the de {B}ruijn-{N}ewman constant},
	url = {https://doi.org/10.1007/BF01385647},
	volume = {58},
	year = {1991}}

@article{rodgers2020t,
	author = {B. Rodgers and T. Tao},
	doi = {10.1017/fmp.2020.6},
	fjournal = {Forum of Mathematics. Pi},
	journal = {Forum Math. Pi},
	mrclass = {11},
	mrnumber = {4089393},
	pages = {e6},
	title = {The de {B}ruijn--{N}ewman constant is non-negative},
	url = {https://doi.org/10.1017/fmp.2020.6},
	volume = {8},
	year = {2020}}
\end{document}